\documentclass[a4paper,11pt]{scrartcl}
\usepackage{mathtools}
\mathtoolsset{showonlyrefs}
\usepackage{amsmath}
\usepackage{amssymb}
\usepackage{amsthm}
\usepackage{amsxtra}
\usepackage{fontenc}
\usepackage[latin1]{inputenc}
\usepackage{graphicx}
\usepackage{pgf}
\usepackage{tikz,pgfkeys,xargs}
\usepackage{pgfplots}
\pgfplotsset{compat=newest}
\pgfplotsset{plot coordinates/math parser=false}
\newlength\figureheight
\newlength\figurewidth 
\usepackage[hang]{subfigure}
\usepackage{multirow}
\usepackage[ruled]{algorithm2e}
\usepackage{etex}
\usepackage{url}

\SetKw{KwInit}{Initialization:}
\SetKw{KwBreak}{break}
\SetKw{KwInput}{Input:}
\SetKw{KwOutput}{Output:}
\SetKw{KwIter}{Iteration:}

\newcommand{\e}{\ensuremath{{\,\rm {e}}}}
\renewcommand{\atop}[2]{\genfrac{}{}{0pt}{}{#1}{#2}}
\newcommand{\dx}[1][x]{\,\mathrm{d}#1}

\def\argmin{\mathop{\rm argmin}}

\newcommand{\R}{\ensuremath{\mathbb{R}}}

\newcommand{\C}{\ensuremath{\mathbb{C}}}

\renewcommand{\atop}[2]{\genfrac{}{}{0pt}{}{#1}{#2}}
\newcommand{\zb}[1]{\ensuremath{\boldsymbol{#1}}}

\newcommand{\imag}{\mathrm{i}}

\def\argmin{\mathop{\rm argmin}}
\newcommand{\tT}{{\scriptscriptstyle \operatorname{T}}}

\DeclareMathOperator{\diag}{diag} 

\DeclareMathOperator{\supp}{supp}

\DeclareMathOperator{\lev}{lev}

\newtheorem{theorem}{Theorem}
\newtheorem{proposition}[theorem]{Proposition}
\newtheorem{lemma}[theorem]{Lemma}
\newtheorem{remark}[theorem]{Remark}

\def\mirr{{\rm mirr}}
\def\zero{{\rm zero}}
\def\per{{\rm per}}
\def\dd{{\rm d}}
\def\prox{{\rm prox}}
\def\diag{{\rm diag}}

\newenvironment{customlegend}[1][]{%
	\begingroup
	\csname pgfplots@init@cleared@structures\endcsname
	\pgfplotsset{#1}%
}{%
\csname pgfplots@createlegend\endcsname
\endgroup
}%

\def\addlegendimage{\csname pgfplots@addlegendimage\endcsname}

\pgfkeys{/pgfplots/number in legend/.style={%
		/pgfplots/legend image code/.code={%
			\node at (0.295,-0.0225){#1};
		},%
	},
}
\begin{document}
\author{
Jan Henrik Fitschen, Friederike Laus and Gabriele Steidl \\
\small Department of Mathematics,
 University of Kaiserslautern, Germany \\ \small
  \{fitschen, friederike.laus, steidl\}@mathematik.uni-kl.de
	}

\title{ 
Transport between RGB Images Motivated by Dynamic Optimal Transport}
\date{\today}

\maketitle

\begin{abstract}
	\noindent
	We propose two models for the interpolation between RGB images 
	based on the dynamic optimal transport model of Benamou and Brenier~\cite{BB00}.
	While the application of dynamic optimal transport 
	and its extensions to unbalanced transform
	were examined for gray-values images in various papers, this
	is the first attempt to generalize the idea to color images.
	The nontrivial task to incorporate color into the model is tackled by considering
	RGB images as three-dimensional arrays, where the transport
	in the RGB direction is performed in a periodic way.
	Following the approach of Papadakis et al.~\cite{PPO14} for gray-value images
	we propose two discrete variational models, a constrained and a penalized one
	which can also handle unbalanced transport.
	We show that a minimizer of our discrete model exists, but it is not unique for some special initial/final images.
	For minimizing the resulting functionals we apply a primal-dual algorithm. 
	One step of this algorithm requires the solution of a four-dimensional discretized Poisson equation 
	with various boundary conditions in each dimension.
	For instance, for the penalized approach we have
	simultaneously zero, mirror and periodic boundary conditions. 
	The solution can be computed efficiently
	using fast Sin-I, Cos-II and Fourier transforms.
	Numerical examples demonstrate the meaningfulness of our model. 
\end{abstract}
\section{Introduction} \label{sec:intro}
%
%
\begin{figure*}	
\centering
	{\includegraphics[width=.4\textwidth]{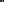}}\hspace{0.8cm}
		{\includegraphics[width=.45\textwidth]{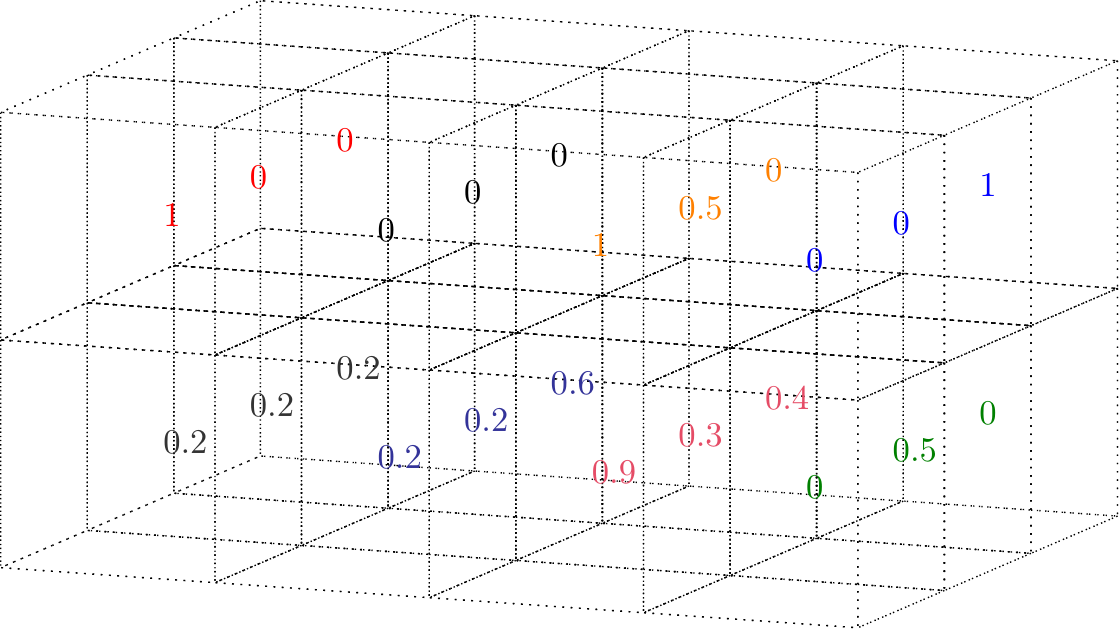}}
	\caption{
	\label{Fig:illu}
		Left: RGB image of size $4 \times 2$. Right: Image as three-dimensional $4 \times 2 \times 3$ array.
		The ``mass'' is the sum of the RGB values of all pixels.
		}		
\end{figure*}
Color image processing is much more challenging compared to gray-value image processing and usually, approaches 
for gray-value images do not generalize in a straightforward way to color images. For example, one-dimensional 
histograms are a very simple, but powerful tool in gray-value image processing, while it is in general difficult to exploit
histograms of color images. In particular, there exist several possibilities to represent color images \cite{Ber14}. 
In this paper we consider the interpolation  between two color images in the RGB space (see Figure~\ref{Fig:curves} (left)) motivated 
by the fluid mechanics formulation of dynamic optimal transport
by Benamou and Brenier~\cite{BB00} and the recent approaches of Papadakis et al.~\cite{PPO14} and \cite{MRSS14} for gray-value images.
In these works  gray-value images are interpreted as
two-dimensional, finitely supported density functions $f_0$ and $f_1$ of  absolutely continuous 
probability measures $\mu_0$ and $\mu_1 $ (i.e. $\mu_i(A)=\int_A f_i \dx$, $i=0,1$).  
In particular, we have $\int_{\mathbb R^2} f_0 \dx = \int_{\mathbb R^2} f_1 \dx = 1$.
Therewith, intermediate images are obtained as the densities $f_t$ 
of the geodesic path $\mathrm{d}\mu_t=f_t \dx$ with respect to the Wasserstein distance between $\mu_0$ and $\mu_1$. 
\\
In this paper, we  extend the transport model to  discrete RGB color images.
The incorporation of color into the above approach appears to be a non trivial task and 
this paper is a first proposal in this direction. 
We consider $N_1 \times N_2$ RGB images as three-dimensional arrays in $\mathbb R^{N_1,N_2,3}$, 
where the third direction is the ``RGB direction'' that contains the color information;
for an illustration see Figure \ref{Fig:illu}.
Particular attention has to be paid to this direction and its boundary conditions.
We propose to use periodic boundary conditions, which is motivated as follows:
assume we are given two color pixels $f_0$ and $f_1\in \R^{1\times 1\times 3}$. 
Using mirror (Neumann) boundary conditions in the third dimension, 
the transport of, e.g., a red pixel  $f_0 = (1,0,0)$ into a blue one 
$f_1 = (0,0,1)$ goes over $(0,1,0)$ (green), see Figure~\ref{Fig:curves} (middle), which is not intuitive from the viewpoint of human color perception. 
Furthermore, it implies that the transport depends on the order of the three color channels, which is clearly not desirable. As a remedy, we suggest to use 
periodic boundary conditions in the color dimension. In case of a red and a blue pixel, it yields a transport via violet,
which is also what one would expect, compare Figure~\ref{Fig:curves} (right) and see also Figures~\ref{fig:boundary} and \ref{Fig:boundary_echt}.
\\
\begin{figure*}	
	{\includegraphics[width=.33\textwidth]{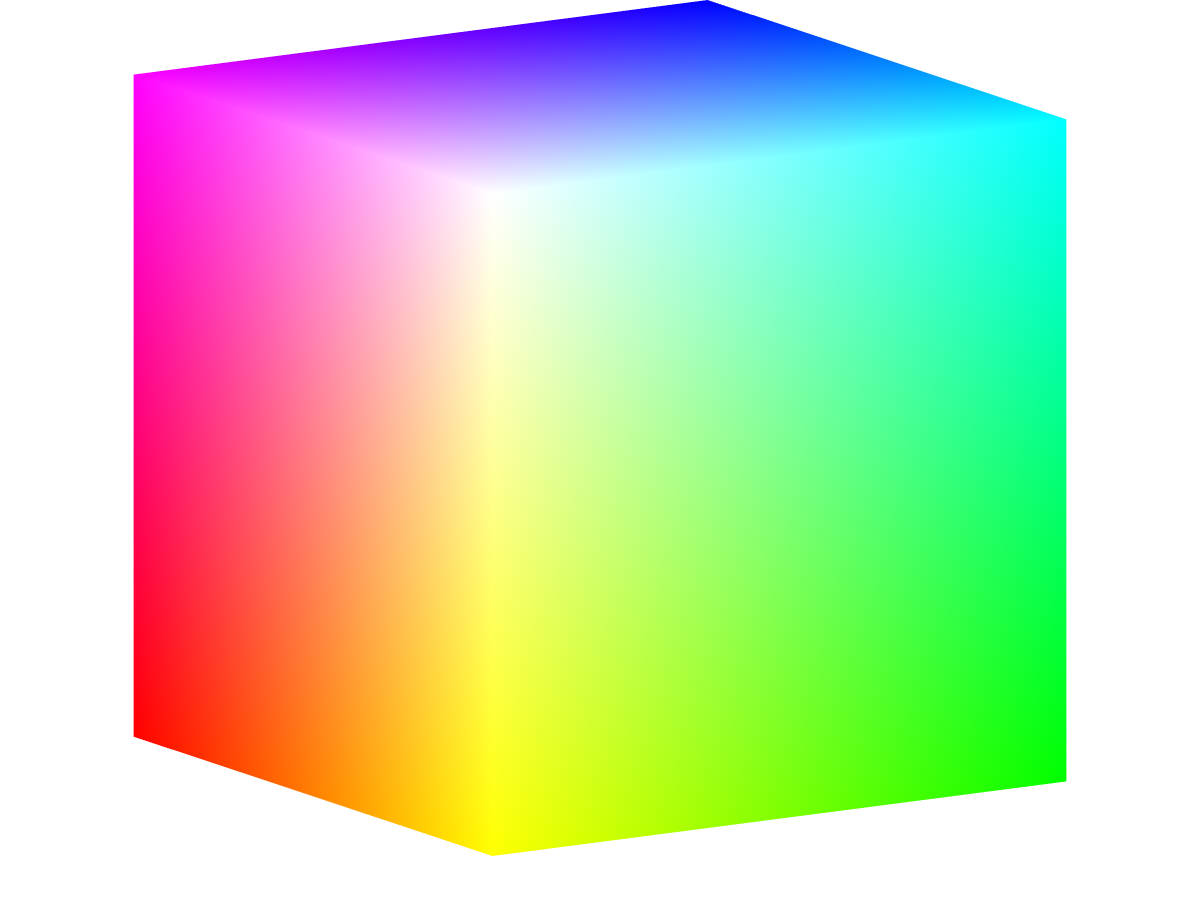}}
	{\includegraphics[width=.3\textwidth]{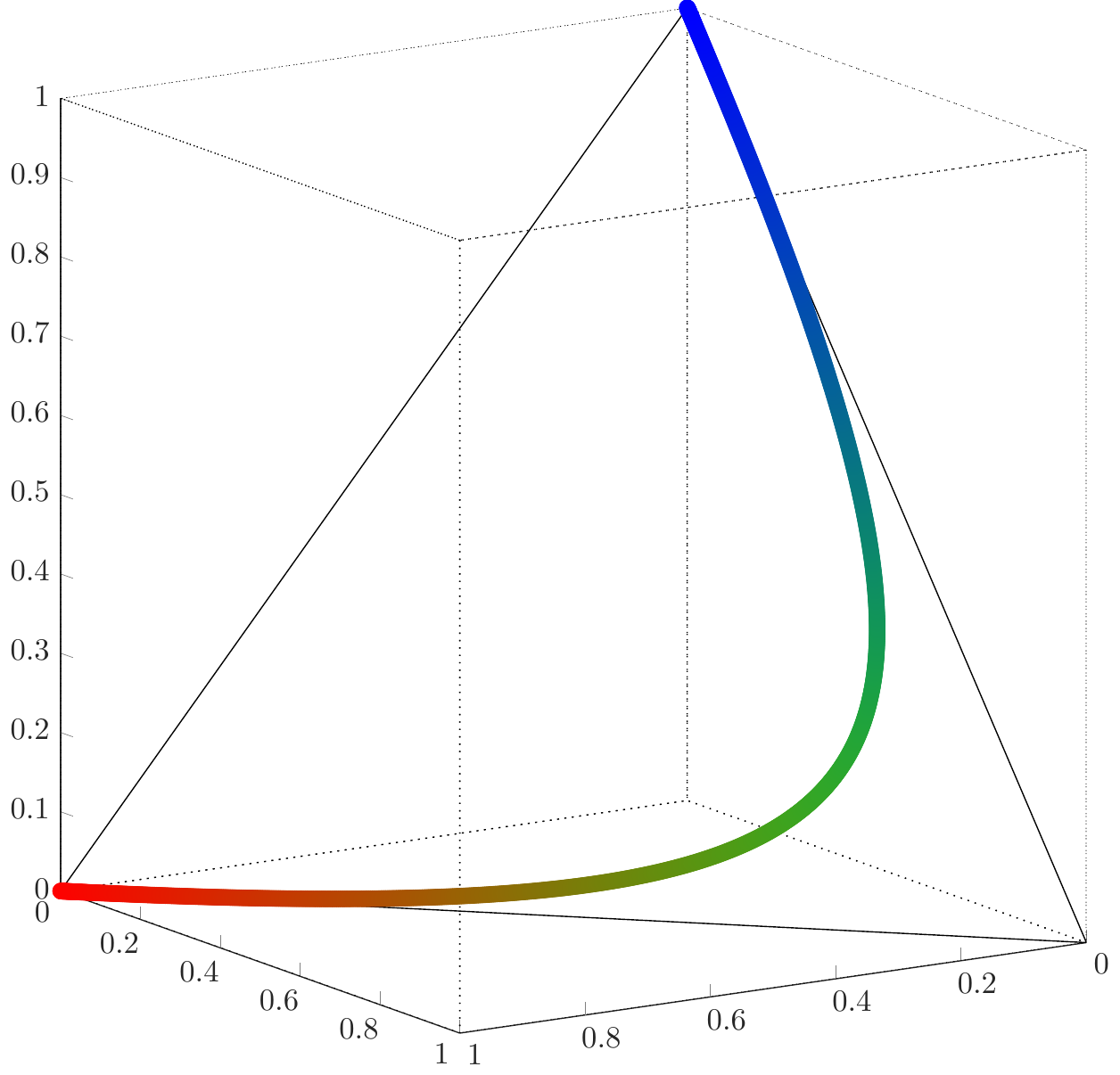}}
	{\includegraphics[width=.3\textwidth]{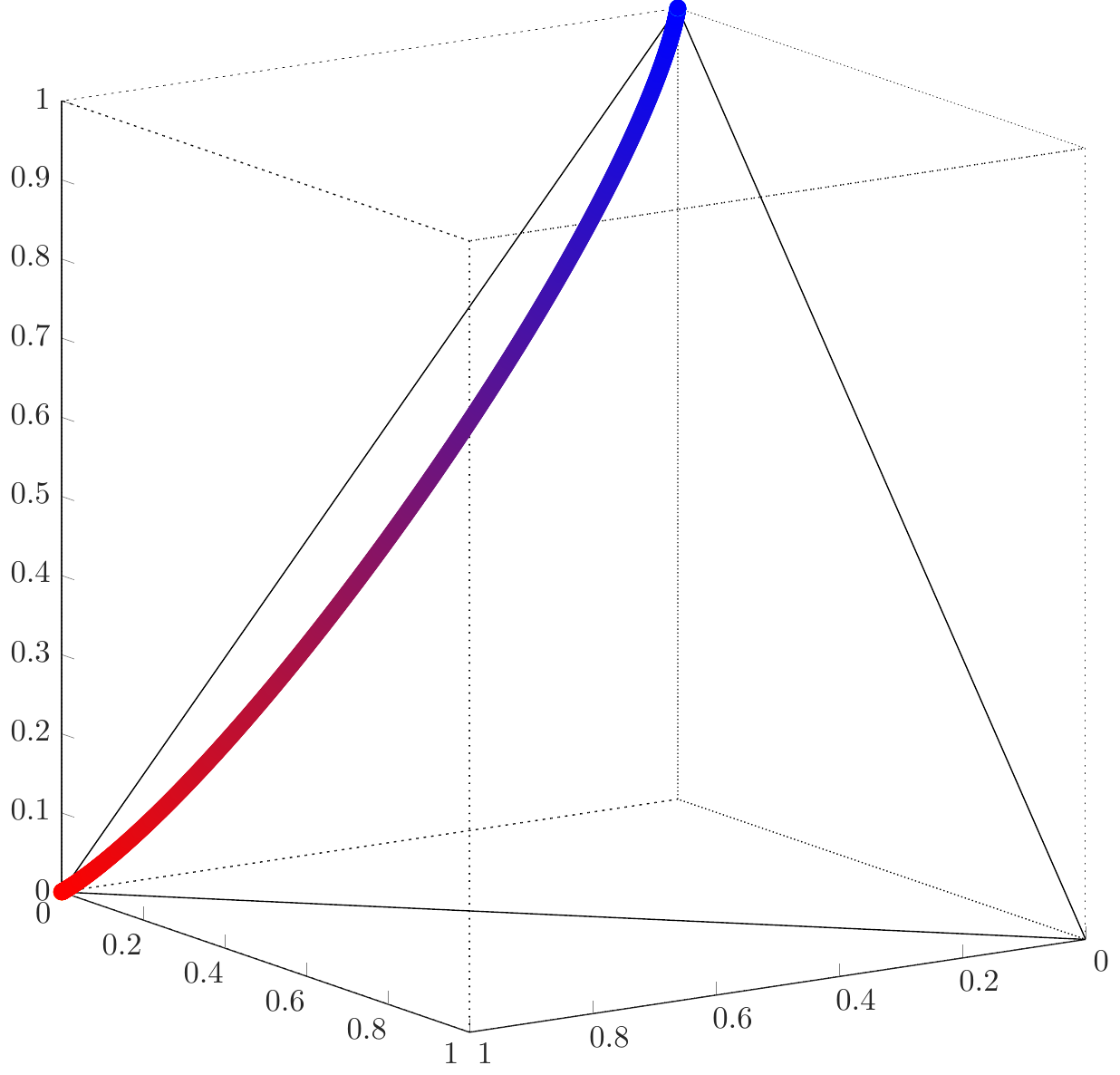}}
	\caption{\label{Fig:curves}
		Left: RGB color cube. Middle/Right:
		Color transfer between $(1,0,0)$ (red) and $(0,0,1)$ (blue) for mirror (middle) and periodic (right) boundary conditions visualized in the RGB cube.}		
\end{figure*}
In the following, we propose two variational models for the transport of color images. To handle also the case of unequal masses the mass conservation constraint is relaxed. 
Our first model contains the continuity equation as a constraint.
It turns out that it generalizes the technique which was proposed in~\cite{PPO14} for gray-value images.
The second model  penalizes the continuity equation, similar as it was also considered for (continuous) gray-value images
in~\cite{MRSS14}.
Other approaches to unbalanced optimal transport can e.g. be found in \cite{Ben03,CSPV15,FPPA14,FZMAP15}.\\
The interpolation based on an optimal transport model for a special class of images, namely so called microtextures,  has been addressed 
by Rabin et al. in \cite{PFR12,RPDB12}.
The authors show that microtextures can be well modeled  as a realization of a Gaussian random field.
In this case, theoretical results guarantee that the intermediate measures $\mu_t$ are Gaussian as well and they can be stated explicitly in terms of the means and covariance matrices learnt from the given images $f_0$ and $f_1$.
The idea can be generalized for interpolating between more than two microtextures by
using barycentric coordinates.
The approach fails for some special microtexture settings which usually do not appear in practice.
\\
Color interpolation between images of the same shape can be realized by switching to the
HSV or HSI space. Then only the hue component has to be transferred using, e.g. by a dynamic extension of the model for the transfer of cyclic measures
(periodic histograms) of Delon et al.~\cite{Del04,RDG11}. For an example we refer to \cite{FLS15}.
This idea is closely related to the affine model for color image enhancement in \cite{NS14a}.
However, these approaches transfer only the color and leave the original edge structure of the image untouched.\\
Finally we mention that the interpolation between images can also be tackled by 
other sophisticated techniques such as metamorphoses, see \cite{TJ05}.
These approach\-es are beyond the scope of this paper.
A combination of the optimal dynamic transport model with the metamorphosis approach was proposed in~\cite{MRSS14}.
\\

The outline of our paper is as follows: 
in Section~\ref{sec:model_cont} we recall basic results from the theory of optimal transport.
At this point, we deal with general $p \in (1,2]$ instead of just $p=2$ as in~\cite{PPO14}.
We propose  discrete dynamic transport models in Section~\ref{sec:model_discrete}.
Here we prefer to give a matrix-vector notation of the problem to make it more intuitive from an linear algebra point of view.
We prove the existence of a minimizer and show that there are special settings where the minimizer is not unique.
In Section \ref{sec:alg} we solve the resulting minimization problems by primal-dual minimization algorithms.
It turns out that one step of the algorithm requires the solution of a 
four-dimensional Poisson equation which
includes various boundary conditions and can be handled by
fast trigonometric transforms.
Another step involves to find the positive root of a polynomial of degree $2q-1$, where 
$\frac1p + \frac1q =1$ and $p \in (1,2]$. For this task we propose to use Newton's algorithm and determine an appropriate starting point to ensure its quadratic convergence. 
Section \ref{sec:numerics} shows numerical results, some of which were also reported at the SampTA conference 2015 \cite{FLS15}. More examples can be found on our website { \small
	\url{http://www.mathematik.uni-kl.de/imagepro/members/laus/color-OT}}.
Finally, Section \ref{sec:conclusions} contains conclusions and ideas for future work.
In particular, additional priors may be used to improve the dynamic transport, e.g., a total variation prior to avoid smearing effects.
The Appendix reviews the diagonalization of certain discrete Laplace operators, and provides basic rules for tensor product computations.
Further it contains some technical proofs.
%
\section{Dynamic Optimal Transport} \label{sec:model_cont}
In this section we briefly review some basic facts on the theory of optimal transport.
For further details we refer to, e.g.~\cite{AGS08,San15,Vil08}.\\
Let $\mathcal{P}(\R^d)$ be the space of probability measures on $\R^d$ and 
$\mathcal{P}_p(\R^d)$, $p \in [1,\infty)$ the Wasserstein space of measures having finite $p$-th moments
$$
{\cal P}_p(\mathbb R^{d}) \vcentcolon= \left\{ \mu \in {\cal P}(\mathbb R^{d}): 
\int_{\mathbb R^{d}} |x|^p d \mu(x) < +\infty\right\}.
$$ 
For $\mu_0, \mu_1\in \mathcal{P}(\R^d)$ let $\Pi(\mu_0,\mu_1)$ be the set of all probability measures on $\R^d\times \R^d$ 
whose marginals are equal to $\mu_0$ and $\mu_1$. 
Therewith, the optimal transport problem (Kantorovich problem) reads as 
$$
\argmin_{\pi\in \Pi(\mu_0,\mu_1) } \int_{\mathbb R^{d}} |x-y|^p \dx[\pi(x,y)].
$$
One can show that for $p \in [1,\infty)$ a minimizer exists, which is uniquely determined for $p>1$ and also called optimal transport plan. 
In the special case of the one-dimensional optimal transport problem, if the measure $\mu_0$ is non-atomic, the optimal transport plan is the same for all $p \in (1,\infty)$ and can be stated explicitly in terms of the cumulative density functions of the involved measures. 
The minimal value
$$
W_p(\mu_0,\mu_1) := \left(\min_{\pi\in \Pi(\mu_0,\mu_1) } \int_{\mathbb R^{d}} |x-y|^p \dx[\pi(x,y)]\right)^{\frac{1}{p}}
$$
defines a distance on $\mathcal{P}_p(\R^d)$, the so-called 
\emph{Wasserstein distance}. \\ 
Wasserstein spaces $\left({\cal P}_p(\R^d), W_p(\R^d)\right)$ are geodesic spa\-ces.
In particular, there exists for any $\mu_0,\mu_1 \in {\cal P}_p(\mathbb R^{d})$ a geodesic
$\gamma\colon [0,1] \rightarrow  {\cal P}_p(\mathbb R^{d})$
with $\gamma(0) = \mu_0$ and $\gamma(1) = \mu_1$.	
For interpolating our images we ask for $\mu_t = \gamma(t)$, $t \in [0,1]$.\\
At least theoretically there are several ways to compute $\mu_t$.
If  the optimal transport plan $\pi$ is known, then 
$\mu_t = {\cal L}_t {}_\# \pi \vcentcolon= \pi \circ {\cal L}_t^{-1}$ yields the geodesic path,
where ${\cal L}_t\colon \mathbb R^{d} \times \mathbb R^{d} \rightarrow \mathbb R^{d}$, 
${\cal L}_t(x,y) = (1-t)x + ty$ is the linear interpolation map, see further~\cite{San15}.
This requires the knowledge of the optimal transport plan $\pi$ and of  ${\cal L}_t^{-1}$.
At the moment there are efficient ways for computing the optimal transport plan 
$\pi$ only in special cases, 
in particular in the one-dimensional case by an ordering procedure and for 
Gaussian distributions in the case 
$p=2$ using  expectation and covariance matrix. 
For $p=2$ one can also use the fact that $\pi$ is indeed induced 
by a transport map $T\colon \mathbb R^d \rightarrow \mathbb R^d$, i.e.,
$\pi = ({\rm id}, T)_{\#}\mu_0$, which can be written as $T = \nabla \psi$, where
$\psi$ fulfills the Monge-Amp\`{e}re equation, see \cite{Caf90a}.
However, this second order nonlinear elliptic PDE
is numerically hard to solve and so far, only some special cases were considered
\cite{Ben95,Cu89,CP84,KO97}.
Other numerical techniques to compute optimal transport plans have been proposed, e.g., 
in \cite{AHT03,HRT10,SS13a,Sch15}.
Another approach consists in relaxing the condition of minimizing a Wasserstein distance by 
using instead an entropy regularized Wasserstein distance. 
Such distances can be computed more efficiently by the Sinkhorn algorithm
and were applied within a barycentric approach by Cuturi et al. \cite{Cu13,CD14}.
\\

The approach in this paper was inspired by the one of Benamou and Brenier in~\cite{BB00}.
It involves the velocity field $\mathrm{v}\colon [0,1]\times \R^d\to \R^d$ of the geodesic curve joining $\mu_0$ and $\mu_1$.
This velocity field ${\rm v}(t, \cdot)$ 
has constant speed
$\|{\rm v}(t, \cdot)\|_{L^p(\mu_t)} = W_p(\mu_0,\mu_1)$.
It can be shown that it minimizes the energy functional 
\begin{align} \label{energy_1}
{\cal E}_p({\rm v},\mu) \vcentcolon= \int_{0}^{1}  \int_{\mathbb R^{d}} \frac1p |{\rm v}(x,t)|^p \dx[\mu_t(x)]\dx[t]
\end{align}
and fulfills the continuity equation
\begin{align}
\partial_t \mu_t + \nabla_x \cdot (\mu_t {\rm v}(t, \cdot)) = 0,
\end{align}
where we say that $t \mapsto \mu_t$ is a measure-valued solution of the continuity equation if for all compactly supported
$\phi\in C^{1}\bigl((0,1)\times \mathbb R^d\bigr)$ and $T\in (0,1)$ the relation
\begin{equation*}
\int_{0}^{T} \int_{\mathbb R^d} \partial_t \phi(x,t) + \langle {\rm v}(x,t),\nabla_x \phi(x,t)\rangle  \dx[\mu_t(x)]\dx[t] = 0
\end{equation*}
holds true. For more details we refer to~\cite{San15}.\\
Assuming the measures $\mu_0$ and $\mu_1$ to be absolutely continuous with respect to the Lebesgue measure, i.e.\ $\mathrm{d}\mu_i = f_i \dx$, $i=0,1$, theoretical results (see for instance~\cite[Theorem 8.7]{Vil08}) guarantee that the same holds true for $\mathrm{d}\mu_t = f_t \dx$, where $f_t$ can be obtained as the minimizer over ${\rm v},f$ of	
\begin{align} \label{energy_2}
{\cal E}_p({\rm v},f)=
\int_0^1\int_{\mathbb R^d} \frac1p |{\rm v}(x,t)|^p f(x,t) \, \dx \dx[t]
\end{align}
subject to the continuity equation
\begin{align}
\partial_t f(x,t) + \nabla_x \cdot ({\rm v}(x,t) f(x,t)) = 0,\\
f(0,\cdot)=f_0, \; f(1,\cdot)=f_1,
\end{align}
where we suppose  $\cup_{t \in [0,1]}\supp f(t,\cdot) \subseteq [0,1]^d$ 
with appropriate (spatial) boundary conditions.
Unfortunately, the energy functional \eqref{energy_2} is not convex in $f$ and ${\rm v}$. 
As a remedy, Benamou and Brenier suggested in the case $p=2$ a change of variables 
$
(f,{\rm v})\mapsto (f,f{\rm v}) = (f,m).
$ 
This idea can be generalized to  $p \in (1,\infty)$, see~\cite{San15}, which results in the functional 
\begin{align} \label{energy_3}
\int_{0}^{1}\int_{\R^d} J_p\bigl(m(x,t),f(x,t)\bigl)\dx \dx[t],\\
\end{align}
where
$J_p\colon \mathbb R^d \times \mathbb R \rightarrow \mathbb R \cup \{+ \infty\}$
is defined as
\begin{align} \label{def_J_p}
J_p (x,y) \vcentcolon= 
\begin{cases}
\frac1p \frac{|x|^{p}}{y^{p-1}} &  {\rm if } \; y>0,\\
0 & {\rm if } \; (x,y)=(0,0),\\
+\infty &{\rm otherwise}
\end{cases}
\end{align}
and $|\cdot|$ denotes the Euclidean norm.
This functional has to be minimized subject to the continuity equation
\begin{align}
\partial_t f(x,t) + \nabla_x \cdot m(x,t)  = 0, \label{cont}\\
f(0,\cdot)=f_0, \; f(1,\cdot)=f_1, \label{start}
\end{align}
equipped with  appropriate spatial boundary conditions.

\begin{remark} \label{J_p}
	The function 
	$J_p\colon \mathbb R^d \times \mathbb R \rightarrow \mathbb R \cup \{+ 
	\infty\}$
	defined in \eqref{def_J_p}
	is the perspective function of $\psi(s) = \frac{1}{p} |s|^p$, i.e.,
	$J_p (x,y) = y \psi\left(\frac{x}{y} \right)$. 
	For properties of perspective functions see, e.g., {\rm \cite{DM08}}.
	In particular, since $\psi$ is convex for $p \in (1,\infty)$, 
	its perspective $J_p (x,y)$ is also convex.
	Further, $J_p (x,y)$ is lower semi-continuous and positively homogeneous,
	i.e. 
	$J_p (\lambda x, \lambda y) = \lambda J_p (x,y)$ for all $\lambda > 0$.
\end{remark}
%
\section{Discrete Transport Models} \label{sec:model_discrete}
%

In practice we are dealing with discrete images whose pixel values are given on a rectangular grid.
To get a discrete version of the minimization problem we have 
to discretize both the integration operator in \eqref{energy_3} by a quadrature rule
and the differential operators in the continuity equation~\eqref{cont}.
The discretization of the continuity equation requires the evaluation of 
discrete ``partial derivatives'' in time as well as in space. 
In order to avoid solutions suffering from the well known checker\-board-effect 
(see for instance~\cite{Pat80}) 
we adopt the idea of a staggered grid as in \cite{PPO14}, see 
Figure~\ref{Fig:staggered_grid}.
%
\begin{figure*}[htbp]
	\centering
	\begin{tikzpicture}[scale=6]\centering	
	\draw[step=0.25cm,very thin, dotted] (-0.1,-0.1) grid (1.1,1.1);
	\draw[very thick,->] (0,1) -- (1.15,1) node[below] {$t$};
	\draw[very thick,->] (0,1) -- (0,-0.15) node[left] {$x$};
	\draw (0,0) node[anchor=north east] {$1$};  
	\draw (1,1) node[above] {$1$};  
	\draw (0,1) node[left] {$0$};     
	
	\foreach \y in {0.125,0.375,0.625,0.875}	{\foreach \x in {0.125,0.375,0.625,0.875}				
		\draw[thick] (\x,\y) ++(-0.015,0)--++(0.03,0)++(-0.015,-0.015)--++(0,0.03);}		 
	
	\foreach \y in {0.25,0.5,0.75}	{ \foreach \x in {0.125,0.375,0.625,0.875}  {  
			\fill (\x,\y) ++ (-0.01,-0.01) rectangle ++ (0.02,0.02); }  } 
	
		\foreach \y in {0,1}	{ \foreach \x in {0.125,0.375,0.625,0.875}  {  
				\draw (\x,\y) ++ (-0.01,-0.01) rectangle ++ (0.02,0.02); }  } 
	
	\foreach \y in {0.125,0.375,0.625,0.875}	{ \foreach \x in {0.25,0.5,0.75}  {  
			\fill (\x,\y) circle (0.01); }  }

	\foreach \y in {0.125,0.375,0.625,0.875}
	\draw (0,\y)circle (0.01); 
	
	\foreach \y in {0.125,0.375,0.625,0.875}
	\draw (1,\y)circle (0.01); 
	
	\begin{customlegend}[legend cell align=left,
	legend entries={ 
		boundary nodes for $f_0$ and $f_1$,
		inner nodes for $f$,
		boundary nodes for $m$,
		inner nodes for $m$,
		interpolation nodes},
	legend style={at={(1.9,0.85)},font=\footnotesize}] 
	\addlegendimage{mark=*,draw=black, fill=white,only marks}
	\addlegendimage{mark=*,draw=black,only marks}
	\addlegendimage{mark=square*,draw=black,fill=white,only marks}
	\addlegendimage{mark=square*,draw=black,fill=black,only marks}
	\addlegendimage{mark=+,black,only marks}
	\end{customlegend}
	
	\end{tikzpicture}
	\caption{Staggered grid for the discretization of the 
	dynamic optimal transport problem, where $N=P=4$. For periodic boundary conditions, the boundary values  for $m$ are equal, while they are zero for Neumann boundary conditions.}\label{Fig:staggered_grid}
\end{figure*}

The differential operators in space and time are discretized by forward differences, 
and depending on the boundary conditions this results in the use of difference 
matrices of the form

{\scriptsize
\begin{align*}
D_n \vcentcolon &= n \left( 
\begin{array}{rrrrrrr}
-1&\; 1&                 \\
  &-1  & \; 1&                \\
  &    &     &\ddots  &    \\
  &    &     &        &-1&1& 0\\
  &    &     &        &  &-1& 1
\end{array} 
\right) \in \R^{n-1,n},
\\
D_n^\per \vcentcolon&= 
n \left( 
\begin{array}{rrrrr}
1&  &    &  &-1 \\
 -1&1&          \\
  & &\ddots& & \\
  & &      &1& 0\\
  & &      &-1&1
\end{array}
\right) \in \R^{n,n}.
	\end{align*}}
For the integration we apply a simple midpoint rule.
To handle this part, we introduce the averaging/inter\-polation matrices

{\scriptsize
\begin{align*}
S_n \vcentcolon &= 
\frac12 \left( 
\begin{array}{rrrrrr}
1&1&              \\
 &1&1&            \\
 & & &\ddots& &   \\
 & & &      &1&1
\end{array}
\right) \in \R^{n-1,n},
\\
S_n^{\per} \vcentcolon&= 
\frac12 \left( 
\begin{array}{rrrrrrr}
1&0& &      & & &1\\
1&1& &            \\
& 1&1& &            \\
 & & &\ddots& & &   \\
 & & &      &1&1&0\\
 & & &      & &1&1
\end{array}
\right) \in \R^{n,n}.
\end{align*}
}

\noindent
{\bf Discretization for one spatial dimension + time}:
In the following, we derive the discretization of \eqref{energy_3} for one spatial direction, i.e., for the transport of signals. 
The problem can be formulated in a simple matrix-vector form using tensor products of matrices
which makes it rather intuitive from the linear algebra point of view.
Moreover, it will be helpful for deriving the fast trigonometric transforms which will play a role within our algorithm.
The generalization of our approach to higher dimensions is straightforward and can be found in Appendix \ref{sec:appC}.
\\
We want to organize the transport between two given one-dimensional, 
nonnegative discrete signals
$$
f_0 \coloneqq \left(f_0( \tfrac{j-1/2}{N} ) \right)_{j=1}^N \quad {\rm and} 
\quad
f_1 \coloneqq \left(f_1( \tfrac{j-1/2}{N} ) \right)_{j=1}^N.
$$
We are looking for the intermediate signals 
$f_t$ for $t = \frac{k}{P}$, $k=1, \ldots, P-1$.
Using the notation $f_t(x) = f(x,t)$, we want to find
$f(\frac{j-1/2}{N},\frac{k}{P})_{j=1,k=1}^{N,P-1} \in \R^{N,P-1}$.
For $m$  we have to take the boundary conditions into account.
 In the case of mirror boundary conditions, there is no flow over the boundary 
and since $m = f\rm v$ the value of $m$ is zero at the boundary at each time.
In the periodic case  both boundaries of $m$ coincide.
The values of $m$ are taken at the cell faces 
$\frac{j}{N}$, $j = \kappa,\ldots,N-1$
and time 
$\frac{k-\tfrac12}{P}$, $k=1,\ldots,P$, i.e., we  are looking for
$\left( m(\frac{j}{N},\frac{k-\tfrac12}{P}) \right)_{j=\kappa,k=1}^{N-1,P} \in \R^{N-\kappa,P}$,
where 
$$
\kappa = 
\left\{
\begin{array} {ll}
1 & \mbox{mirror boundary},\\
0 & \mbox{periodic boundary}.
\end{array}
\right.
$$
The midpoints for the quadrature rule 
are computed by averaging the neighboring two values of $m$ and $f$, respectively.
To give a sound matrix-vector notation of the discrete minimization problem
we reorder $m$ and $f$ columnwise into vectors
${\rm vec}(f) \in \R^{N(P-1)}$ and 
${\rm vec} (m) \in \R^{(N-\kappa)P}$, which we again denote by $f$ and $m$. 
For the ${\rm vec}$ operator in connection with the tensor product $\otimes$ of matrices we refer to Appendix \ref{sec:appB}.
More specifically, let $I_n\in\R^{n,n}$ be the identity matrix and set
\begin{align}
S_{\rm f} \vcentcolon= S_P^ \tT \otimes I_N, \qquad D_{\rm f} \vcentcolon= D_P^\tT \otimes I_N,
\end{align}
\begin{align}
S_{\rm m} \vcentcolon= \left\{
\begin{array}{ll}
I_P \otimes S_N^\tT & \: \mbox{mirror boundary},\\
I_P \otimes \left( S_{N}^{\per} \right)^\tT & \: \mbox{periodic boundary},
\end{array}
\right.
\end{align}
\begin{align}
D_{\rm m} \vcentcolon=
\left\{
\begin{array}{ll}
I_P \otimes D_N^\tT & \mbox{mirror boundary},\\
I_P \otimes \left( D_{N}^{\per} \right)^\tT & \mbox{periodic boundary}.
\end{array}
\right. 
\end{align}
Finally, we introduce the  vectors
\begin{align}
f^+ \vcentcolon= \frac12 \bigl( f_0^\tT, \zb 0,  f_1^\tT \bigr)^\tT
, \qquad
f^- \vcentcolon= P \bigl(- f_0^\tT,\zb  0,  f_1^\tT \bigr)^\tT,
\end{align}
where we  denote by ${\bf 0}$ (and ${\bf 1}$) arrays of appropriate size 
with entries 0 (and 1). They are used to guarantee that the boundary conditions are fulfilled.
Now the continuity equation \eqref{cont} 
together with the boundary conditions \eqref{start} for $f$ 
can be reformulated as requirement that  
$(m,f)$ has to lie within the hyperplane
\begin{equation} \label{cont_disc}
\mathcal{C}_0  \coloneqq \left\{ \begin{pmatrix} m\\f \end{pmatrix}\colon 
\underbrace{( D_{\rm m} |  D_{\rm f} )}_{A}
\begin{pmatrix} m\\f \end{pmatrix} = f^- \right\}.
\end{equation}
We will see in Proposition \ref{moore-penrose-constr} that $A A^\tT$ is rank one deficient.
Since further $\zb 1^\tT A = \zb 0$, we conclude that the under-determined 
linear system in \eqref{cont_disc}
has a solution if and only if $\zb 1^\tT f^- = 0$, i.e., if and only if $f_0$ and $f_1$ have the same mass
\begin{equation} \label{continuity_cond}
\zb 1^\tT f_0 = \zb 1^\tT f_1.
\end{equation}

This resembles the fact that dynamic optimal transport is performed between probability measures.
The interpretation of a color image as a probability density function has a major drawback; to represent a valid density, 
the  sum  of all RGB pixel values of the given images, i.e., the sum of the image intensity values, has to be one (or at least equal). 
Therefore,  we consider more general the set
\begin{equation} \label{C}
\mathcal{C}  \coloneqq \argmin_{(m,f)} \| ( D_{\rm m} |  D_{\rm f} )
\begin{pmatrix} m\\f \end{pmatrix} -f^-\|_2^2. 
\end{equation}
Note that the boundary conditions \eqref{start} for $f$ are preserved, while the mass conservation \eqref{continuity_cond}
is no longer required. Clearly, if \eqref{continuity_cond} holds true, then ${\cal C}$ coincides with ${\cal C}_0$.
Let $\iota_{\mathcal{C}}$ denote the indicator function of ${\cal C}$ defined by
\begin{equation} 
\iota_{\mathcal{C}} (x) \coloneqq \left\{
\begin{array}{ll}
0      &{\rm if} \; x \in \mathcal{C},\\
+\infty&{\rm if} \; x \not\in \mathcal{C}.
 \end{array}
\right.
\end{equation}
For $p \in (1,2]$, we consider the following transport problem:
\\[2ex]
{\bf Constrained Transport Problem}:
\begin{equation} \label{model_disc}
\argmin_{ (m,f) }  E(m,f) \coloneqq \| J_p(S_{\rm m} m, S_{\rm f} f + f^+) \|_1 
+ \iota_{\cal C}(m,f).
\end{equation}
Here, the application of $J_p$ is meant componentwise and the summation over its (non-negative) components
is addressed by the $\ell_1$-norm. 
The interpolation operators $S_{\rm m}$ and $S_{\rm f}$ arise from the midpoint rule for computing the integral.\\
We can further relax the relaxed continuity assumption $(m,f) \in {\cal C}$ by replacing it by
$$\| ( D_{\rm m} |  D_{\rm f} )
\begin{pmatrix} m\\f \end{pmatrix} -f^-\|_2^2 \le \tau,
$$
where 
$\tau \ge \tau_0 \coloneqq \min_{(m,f)} \| ( D_{\rm m} |  D_{\rm f} )
( m^\tT,f^\tT)^\tT -f^-\|_2^2$.
For $\tau = \tau_0$ we have again problem \eqref{model_disc}.
Since there is a correspondence between the solutions of such constrained problems with parameter $\tau$
and the penalized problem with a corresponding parameter $\lambda$, see \cite{ABF13,Ro70,TSC13},
we prefer to consider the following penalized problem
with regularization parameter $\lambda >0$:
\\[2ex]
{\bf Penalized Transport Problem}:
\begin{align} \label{model_disc_pen}
\argmin_{ (m,f) } E_\lambda(m,f) \coloneqq \| J_p(S_{\rm m} m, S_{\rm f} f + 
f^+) 
\|_1
+ \lambda \|( D_{\rm m} |  D_{\rm f} )
\begin{pmatrix} m\\f \end{pmatrix} - f^-\|_2^2 .
\end{align}
Note that also for our penalized model the boundary conditions \eqref{start} for $f$ still hold true.
In general, both models \eqref{model_disc} and \eqref{model_disc_pen} do not guarantee that the values of $f$
stay within the RGB cube during the transport.
This is not a specific problem for color images but can  appear for gray-value images as well (gamut problem). 
A usual way out is a final backprojection 
onto the image range. 
An alternative  in the constrained model is
a simple modification of the constraining set in \eqref{C} towards
$
\mathcal{C}  \coloneqq \argmin_{m,f \in [0,1]^3} \| ( D_{\rm m} |  D_{\rm f} )
\begin{pmatrix} m\\f \end{pmatrix} -f^-\|_2^2 
$.
This  leads to inner iterations of the Poisson solver and a projection onto the cube in the subsequent Algorithm 1.
In the penalized problem,  the term $\iota_{[0,1]^3}$ could be added. 
However, we observed in all our numerical experiments only very small violations of the range constraint, which are most likely caused by numerical reasons.\\
A penalized model for the continuous setting and gray-value images was examined in \cite{MRSS14}.
For recent papers on unbalanced transport we refer to \cite{Ben03,CSPV15,FPPA14,FZMAP15}. 
\\
To show the existence of a solution of the discrete transport problems we use the concept
of asymptotically level stable functions.
As usual, for a function $F\colon \mathbb R^n \rightarrow \mathbb R \cup 
\{+\infty\}$ 
and $\mu >\inf_x F(x)$, the level sets are defined by
\[
\lev(F,\mu) \coloneqq \{x \in \R^n: F(x) \leq \mu \}.
\]
By  $F_\infty$ we denote the {\it asymptotic (or recession) function} of $F$ which
according to \cite{Ded77}, see also \cite[Theorem 2.5.1]{AT03}, can be computed by 
	\[
	F_\infty (x) =
	\liminf_{\atop{x'\rightarrow x}{t\rightarrow\infty}}\frac{F(tx')}{t}.
	\]
The following definition of \emph{asymptotically level stable} functions is taken from \cite[p.~94]{AT03}:
a proper and lower semicontinuous function 
$F\colon\R^{n} \rightarrow \R \cup \{+\infty\}$
is said to be {\sl asymptotically level stable}
if for each $\rho>0$,
each real-valued, bounded sequence $\{\mu_k\}_k$ and
each sequence $\{x_k\}_k$ satisfying
\begin{align} \label{aal}
x_k \in \lev(F,\mu_k), \quad \|x_k\|_2 \rightarrow +\infty, \quad
\frac{x_k}{\|x_k\|_2} \rightarrow \tilde x \in \ker (F_\infty),
\end{align}
there exists $k_0$ such that
\begin{equation} \label{lsa}
x_k- \rho \tilde x \in \lev(F,\mu_k)\qquad \text{for all } k\geq k_0.
\end{equation}
If for each real-valued, bounded sequence $\{\mu_k\}_k$ there exists
no sequence $\{x_k\}_k$
satisfying \eqref{aal}, then $F$ is automatically asympto\-ti\-cally level stable.	
In particular, coercive functions are  asymptotically level stable.
It was originally exhibited in \cite{BBGT98} (without the
notion of  asymptotically level stable functions) that any  asymptotically level stable function $F$
with $\inf F>-\infty$ has a global minimizer.
A proof was also given in \cite[Corollary 3.4.2]{AT03}.
With these preliminaries we can prove the existence  of minimizers of our transport models.
%
\begin{proposition}\label{prop:existence}
	The discretized dynamic transport models~\eqref{model_disc} and ~\eqref{model_disc_pen} have a solution.
\end{proposition}
%
\begin{proof}
We show 	that the proper, lower semicontinuous functions $E$ and $E_\lambda$ are asympto\-ti\-cally level stable 
which implies the existence of a minimizer.  
For the penalized problem, the asymptotic function $E_{\lambda,\infty}$ reads 	
\begin{equation*}
	E_{\lambda, \infty}(m,f)=\liminf_{\substack{(m',f')\to (m,f),\\ t\to \infty}}\frac{E_\lambda \bigl(t(m',f')\bigr)}{t}.
\end{equation*}
We obtain
\begin{align*}
\frac{E_\lambda\bigl(t(m',f')\bigr)}{t}& = \frac{1}{t}\left( \| 
J_p\bigl(t(S_{\rm m} m' , S_{\rm f} f'+ \frac{1}{t} f^+)\bigr)\|_1 + \lambda \|( D_{\rm 
m} |  D_{\rm f} )\begin{pmatrix} tm'\\tf' \end{pmatrix} - f^-\|_2^2\right)\\
& = \| J_p(S_{\rm m} m' , S_{\rm f} f'+ \frac{1}{t}f^+)\|_1 + \lambda t\|( 
D_{\rm m} |  D_{\rm f} )\begin{pmatrix} m'\\f' \end{pmatrix} -  
\frac{1}{t}f^-\|_2^2.
\end{align*}
Thus, 	$(\tilde{m},\tilde{f})\in \ker (E_{\lambda,\infty})$ implies 
\begin{align}\label{existence_cond}
 (\tilde{m},\tilde{f})\in \ker(D_{\mathrm{m}}|D_{\mathrm{f}}), \quad \tilde{m} \in \ker (S_{\mathrm{m}}), \quad S_{\mathrm{f}}\tilde{f}\geq 0.
\end{align}
For the constrained problem we have the same implications so that we can restrict our attention  to the penalized one.
By the definition of  $S_{{\rm m}}$ we obtain
${\rm ker } (S_{{\rm m}}) = 
\left\{w \otimes \tilde {\zb 1}: w \in \mathbb R^P\right\}$ for periodic boundary conditions and even $N$ and ${\rm ker } (S_{{\rm m}}) = \{\zb 0\}$  otherwise,
where is defined as $\tilde {\bf 1} = (1,-1,\ldots,1,-1)^\tT\in \mathbb R^N$.
In the case ${\rm ker} (S_{{\rm m}}) = \{\zb 0\}$,
the first and second condition in \eqref{existence_cond} imply 
$D_{{\rm f}} \tilde f = 0$ so that by the definition of $D_{{\rm f}}$ also
$\tilde f = 0$.
In the other case, we obtain by the first condition in \eqref{existence_cond} that
$\tilde f = -D_{{\rm f}}^\dagger D_{{\rm m}} \tilde m$,
where
$
D_{{\rm f}}^\dagger = (D_{{\rm f}}^\tT D_{{\rm f}})^{-1}D_{{\rm f}}^\tT
$
denotes the Moore-Penrose inverse of $D_{{\rm f}}$.
Then
$$
S_{{\rm f}} \tilde f =- S_{{\rm f}} D_{{\rm f}}^\dagger D_{{\rm m}} \tilde m 
= -S_{{\rm f}} D_{{\rm f}}^\dagger D_{{\rm m}} (w \otimes \tilde {\zb 1})
$$
for some $w \in \mathbb R^P$.
Straightforward computation shows
	\begin{equation*}
S_{{\rm f}} \tilde f =-	S_{\rm f} D^\dagger_{\rm f} D_{\rm m} \left( w \otimes \tilde{\zb 1}_N \right) 
=- \tilde w \otimes \tilde{\zb 1}_N 
	\end{equation*}
for some $\widetilde{w}\in \mathbb R^P$.
Now the third condition in \eqref{existence_cond} can only be fulfilled if 
$\tilde w = \zb 0$.
Consequently we have in both cases
\begin{equation} \label{condition_add}
S_{{\rm f}} \tilde f = 0.
\end{equation}
Let $\rho>0$,  $\{\mu_k\}_k$ be a bounded sequence and 
$\{(m_k,f_k)\}_k$ be a sequence fulfilling~\eqref{aal}. 
By \eqref{existence_cond} and \eqref{condition_add} we conclude
	\begin{align}
		E_\lambda \bigl((m_k,f_k)-\rho(\tilde{m},\tilde{f})\bigr)
		 &=\| J_p(S_{\rm m} m_k , S_{\rm f} f_k + f^+) \|_1	+ \lambda \|( D_{\rm m} |  D_{\rm f} )
		\begin{pmatrix} m_k\\f_k \end{pmatrix} - f^-\|_2^2\\
		 &= E_\lambda \bigl((m_k,f_k)\bigr).
	\end{align}
Since $(m_k,f_k)\in \lev(E_\lambda,\mu_k)$, this shows that $(m_k,f_k)-\rho(\tilde{m},\tilde{f})\in \lev(E_\lambda,\lambda_k)$ 
as well and finishes the proof.	
\end{proof}
Unfortunately, 
$J_p(u,v)$ is not strictly convex on its domain as it can be deduced from the following proposition.
%
\begin{proposition} \label{perspective}
For any two minimizers $(m_i,f_i)$, $i=1,2$ of \eqref{model_disc}
the relation
$$\frac{S_{\rm m} m_1}{S_{\rm f} f_1 + f^-} = \frac{ S_{\rm m}m_2}{S_{\rm f} f_2 + f^-}$$
holds true.
\end{proposition}
%
\begin{proof} 
We use the perspective function notation from Remark \ref{J_p}.
For $\lambda \in (0,1)$ and $(u_i,v_i)$ with $v_i >0$, $i=1,2$, we have (componentwise)
\begin{align}
  J_p \left( \lambda (u_1,v_1) + (1-\lambda) (u_2,v_2) \right)
  & =
 \left( \lambda v_1 + (1-\lambda)v_2 \right) 
 \psi \left( \tfrac{\lambda u_1 + (1-\lambda)u_2}{\lambda v_1 + (1-\lambda)v_2} \right)\\
&= 
(\lambda v_1 + (1-\lambda) v_2) \psi 
\left( 
\tfrac{\lambda v_1}{\lambda v_1 + (1-\lambda)v_2} \tfrac{u_1}{v_1}\right. 
\left.
+ \tfrac{(1-\lambda) v_2}{\lambda v_1 + (1-\lambda)v_2} \tfrac{u_2}{v_2}
\right) 
\end{align}
and if 
$\frac{u_1}{v_1} \not = \frac{u_2}{v_2}$
by the strict convexity of $\psi$ that
\begin{align*}
J_p \left( \lambda (u_1,v_1) + (1-\lambda) (u_2,v_2) \right)
 <\lambda J_p(u_1,v_1) + (1-\lambda)J_p(u_2,v_2).
\end{align*}
Setting $u_i \coloneqq S_{\rm m} m_i$ and $v_i \coloneqq S_{\rm f} f_i + f^-$, 
$i=1,2$, we obtain the assertion.
\end{proof}
%
\begin{remark} \label{unique}
For periodic boundary conditions, even $N$ and 
$f_1 = f_0 + \gamma \tilde {\zb 1}$,
$\gamma \in [0,\min f_0 )$ the minimizer of \eqref{model_disc} is not unique. 
This can be seen as follows: Obviously, we would have a minimizer $(m,f)$ if 
$m = w \otimes \tilde {\zb 1} \in {\rm ker}(S_{\rm m})$ for some $w \in \mathbb R^P$
and there exists $f\ge 0$ which fulfills the constraints.
Setting $f^{k/P} \vcentcolon= f(j-1/2,k)_{j=1}^N$, $k=0,\ldots,P$, 
these constraints read
$- 2 P w \otimes \tilde {\zb 1} = P (f^{(k-1)/P} - f^{k/P})_{k=1}^P$. 
Thus, any $w \in \mathbb R^P$ such that
\begin{align} 
	f^{1/P} &= f_0 + 2w_1 \tilde {\zb 1}, \;
	f^{2/P} = f_0 + 2(w_1 + w_2) \tilde {\zb 1}, \ldots \, ,\;
	f^{1} = f_0 + 2(w_1 + w_2+ \ldots+w_P) \tilde {\zb 1}
\end{align}
are nonnegative vectors provides a minimizer of \eqref{model_disc}.
We conjecture that the solution is unique in all other cases,
but have no proof so far.
\end{remark}
%
\section{Primal-Dual Minimization Algorithm} \label{sec:alg}
\subsection{Algorithms} \label{subsec:basicalg}
%
For the minimization of our functionals we apply the primal-dual algorithm 
known as 
Chambolle-Pock algorithm \cite{CP11,PCCB09} in the form of Algorithm 8 in \cite{BSS14}.
We use the following reformulation of the problems: 

{\bf Constrained Transport Problem}:
\begin{align} \label{model_disc_pdhg}
\argmin_{ (m,f) } &\| J_p(u,v) \|_1 + \iota_{\cal C}(m,f) \\
&\mbox{subject to} \quad 
S_{\rm m} m = u,\; S_{\rm f} f + f^+ = v.
\end{align}
%
	\begin{algorithm*}
	{\small	\KwInit{$m^{(0)}= \zb 0$, $f^{(0)}= \zb 0$, 
		$b_m^{(0)}=b_f^{(0)}=\bar b_u^{(0)}=\bar b_v^{(0)}=\zb 0$, 
		$\theta \in (0,1]$,\\ $\tau,\sigma$ with  $\tau \sigma < 1$.} \\
		\KwIter{For $r = 0,1,\ldots$ iterate}
		\begin{align}
			1. \begin{pmatrix} m^{(r+1)} \\ f^{(r+1)} \end{pmatrix}  
			&\vcentcolon= \ \argmin_{(m,f) \in {\cal C}} 
			 \frac{1}{2\tau} \| 
			\begin{pmatrix} m \\ f \end{pmatrix} 
			- \begin{pmatrix} m^{(r)} \\ f^{(r)} \end{pmatrix} +
			\tau \sigma  
			\begin{pmatrix} 
			S_{\rm m}^\tT \bar b_u^{(r)}\\
			S_{\rm f}^\tT \bar b_v^{(r)} 
			\end{pmatrix} \|_2^2
			\\
			2. \ \begin{pmatrix} u^{(r+1)} \\ v^{(r+1)} \end{pmatrix}  
			&\vcentcolon= \ \argmin_{(u,v)} \|{J}_p(u,v)\|_1 
			 + \frac{\sigma}{2} 
			 \|
			\begin{pmatrix} u \\ v \end{pmatrix} 
			 - \begin{pmatrix} 
			S_{\rm m} m^{(r+1)} \\
			S_{\rm f}f^{(r+1)}
			\end{pmatrix}
						-
			\begin{pmatrix} 
			 0 \\
			f_b^+
			\end{pmatrix}
			- 
			\begin{pmatrix} b_u^{(r)} \\ b_v^{(r)} \end{pmatrix} 
			\|_2^2			
			\\
			3. \qquad 
			b_u^{(r+1)} &\vcentcolon= \ b_u^{(r)} + S_{\rm m} m^{(r+1)} - u^{(r+1)} \\
			b_v^{(r+1)} &\vcentcolon= \ b_v^{(r)} + S_{\rm f} f^{(r+1)} + f_b^+ - v^{(r+1)}    		
			\\ 
			4. \qquad 
			\bar b_u^{(r+1)} &\vcentcolon= \ b_u^{(r+1)} + \theta (b_u^{(r+1)}-b_u^{(r)} ) \\
			\bar b_v^{(r+1)} &\vcentcolon= \ b_v^{(r+1)} + \theta (b_v^{(r+1)}-b_v^{(r)} )
		\end{align}
		\caption{Primal-Dual Algorithm for the Constrained Problem \protect\eqref{model_disc_pdhg}} \label{alg1}}%
	\end{algorithm*}
%

{\bf Penalized Transport Problem}:
\begin{align} \label{model_pen_pdhg}
\argmin_{ (m,f) } &\| J_p(u,v) \|_1 + \lambda \|\underbrace{( D_{\rm m} |  
D_{\rm f} )}_{A}
\begin{pmatrix} m\\f \end{pmatrix} - f^-\|_2^2\\
& \mbox{subject to} \quad 
S_{\rm m} m = u,\; S_{\rm f} f + f^+ = v.
\end{align}
%
	\begin{algorithm*}
	{\small	\KwInit{$m^{(0)}= \zb 0$, $f^{(0)}= \zb 0$, 
		$b_u^{(0)}=b_v^{(0)}=\bar b_u^{(0)}=\bar b_v^{(0)}= \zb 0$, 
		$\theta \in (0,1]$,\\ $\tau,\sigma$ with  $\tau \sigma < 1$.} \\
		\KwIter{For $r = 0,1,\ldots$ iterate}
		\begin{align}
			&1. \begin{pmatrix} m^{(r+1)} \\ f^{(r+1)} \end{pmatrix}  
			\vcentcolon= \ \argmin_{(m,f)} \frac{\lambda}{2} \|(D_{\rm m} |  
			D_{\rm f}) \begin{pmatrix} m \\ f \end{pmatrix} - f_b^-\|_2^2 +
			 \frac{1}{2\tau} \| 
			\begin{pmatrix} m \\ f \end{pmatrix} 
			- \begin{pmatrix} m^{(r)} \\ f^{(r)} \end{pmatrix} +
			\tau \sigma  
			\begin{pmatrix} 
			S_{\rm m}^\tT \bar b_u^{(r)}\\
			S_{\rm f}^\tT\bar b_v^{(r)}
			\end{pmatrix} \|_2^2
			 \\
			&2 - 4. \quad \mbox{as in Algorithm \ref{alg1}}
		\end{align}
		\caption{Primal-Dual Algorithm for the Penalized Problem \protect\eqref{model_pen_pdhg}} \label{alg2}}%
	\end{algorithm*}
%

In the following we detail the first two steps of Algorithms~\ref{alg1} and~\ref{alg2}:
\begin{itemize}
\item Step 1 of Algorithm~1 requires the  projection onto ${\cal C}$,
\item Step 1 of Algorithm~2 results in the solution of a linear system of equations 
with
coefficient matrix $\lambda A^\tT A + \frac{1}{\tau} I$ whose Schur complement 
can be computed via fast trigonometric transforms,
\item Step 2 of both algorithms is the  proximal map of $J_p$.
\end{itemize}

\subsection{Projection onto ${\cal C}$}
Step~1 of Algorithm~\ref{alg1} requires to find the 
orthogonal projection of 
$a \coloneqq \begin{pmatrix} m^{(r)} \\ f^{(r)} \end{pmatrix} -
			\tau \sigma  
			\begin{pmatrix} 
			S_{\rm m}^\tT \bar b_u^{(r)}\\
			S_{\rm f}^\tT \bar b_v^{(r)} 
			\end{pmatrix}$
onto ${\cal C}$. 
This means that we have to find a
minimizer of $\|A x - f^-\|_2$ for which $\|a-x\|_2$ attains its smallest value.
Substituting $y \coloneqq a-x$ we are looking
for a minimizer $y$ of $\|A y - A a + f^-\|_2$ with smallest norm $\|y\|_2$.
By~\cite[Theorem 1.2.10]{Bjo96}, this minimizer is uniquely determined by
$A^\dagger (A a - f^-)$.
Therefore the projection of $a$ onto ${\cal C}$ is given by
\begin{align}\label{proj_hyper}
\Pi_{{\cal C}} (a) &= a - A^\dagger \left(A a- f^-\right) \\
&= a - A^\tT (A A^\tT)^\dagger \left(A a- f^-\right). 
\end{align} 
Note that the projection onto ${\cal C}$ coincides with the one onto ${\cal 
C}_0$
if the given images $f_0$ and $f_1$ have the same mass.
The Moore-Penrose inverse of the quadratic matrix $A A^\tT$ is defined as follows:
Let $ A A^\tT$ have the spectral decomposition 
$$
A A^\tT = Q \, \diag(\lambda_j)  \, Q^\tT.
$$ 
Then it holds
\begin{align}
(A A^\tT)^\dagger &=  Q \, \diag(\tilde \lambda_j) \, Q^\tT, \ \text{with }
\tilde \lambda_j \vcentcolon= 
\left\{
\begin{array}{ll}
\tfrac{1}{\lambda_j} &{\rm if} \;  \lambda_j >0,\\
0&{\rm otherwise.} 
          \end{array}
          \right.
\end{align}
The following proposition shows the form of $(A A^\tT)^\dagger$ in the one-dimensional 
spatial case. It appears that the projection onto ${\cal C}$ amounts to solve a two-dimensional Poisson equation 
which can be realized depending on the boundary conditions by fast cosine and Fourier transforms in $\mathcal{O}(NP\log(NP))$ operations.
%
\begin{proposition}  \label{moore-penrose-constr}
Let 
$
C_N \vcentcolon= \sqrt{\tfrac{2}{n}} \left( \epsilon_j \cos\frac{j(2k+1)\pi}{2N} \right)_{j,k=0}^{N-1}
$
with
$\epsilon_0 \vcentcolon= 1/\sqrt{2}$ and $\epsilon_j \vcentcolon= 1$, $j=1,\ldots,N-1$
be the $N$-th cosine matrix and 
$
F_N :=
\sqrt{\tfrac{1}{N}} \left( \mathrm{e}^{\frac{-2\pi \imag jk}{N}} \right)_{j,k=0}^n
$
be the $N$-th Fourier matrix.
Set 
${\rm d}_N^{{\rm mirr}} \vcentcolon= $
and 
$\dd^{\per}_N \vcentcolon= \left( 4 \sin^2 \frac{k\pi}{N}\right)_{k=0}^{N-1}$. 
Then the Moore-Penrose inverse $(A A^\tT)^\dagger$ in~\eqref{proj_hyper} is given by
\begin{align}
(A A^\tT)^\dagger = 
\begin{cases}
(C_P^\tT \otimes C_{N-1}^\tT ) \, {\rm diag}(\tilde \dd ) \, (C_P \otimes C_{N-1}) & {\rm mirror \; boundary},\\ 
(C_P^\tT \otimes \bar F_N ) \, {\rm diag}(\tilde \dd ) \, (C_P \otimes F_N) & {\rm periodic \; boundary},
\end{cases}
\end{align}
where

\begin{align*}
\dd \coloneqq
\begin{cases}
I_P \otimes N^2\diag(\dd_{N-1}^{\mirr} ) + P^2 \diag(\dd_P^{\mirr} ) \otimes I_{N-1}& {\rm mirror \; boundary},\vspace{0.25cm}\\
I_P \otimes N^2\diag(\dd_N^{\per} ) + P^2 \diag(\dd_P^{\mirr} ) \otimes I_N& {\rm periodic \; boundary}
\end{cases}
\end{align*}
and 
$\tilde \dd_j \vcentcolon= \frac{1}{\dd_j}$ if $\dd_j >0$ and $\dd_j=0$ otherwise. 
\end{proposition}
%
The proof is given in Appendix C.
\subsection{Schur Complement of $\lambda A^\tT A + \frac{1}{\tau} I$}
To find the minimizer in Step~1 of Algorithm~\ref{alg2} we set the gradient of the functional to zero 
which results in the solution of the linear system of equations
\begin{align*}
\left(\lambda A^\tT A + \frac{1}{\tau} I\right)\begin{pmatrix} m \\ f 
\end{pmatrix}=
\lambda A^\tT f^- +
\begin{pmatrix} m^{(r)} \\ f^{(r)} \end{pmatrix} -
	\tau \sigma  
	\begin{pmatrix} 
	S_{\rm m}^\tT \bar b_u^{(r)}\\
	 S_{\rm f}^\tT\bar b_v^{(r)} 
	\end{pmatrix} .
\end{align*}
Noting that $\lambda A^\tT A + \frac{1}{\tau} I$ is a symmetric and positive definite matrix, 
this linear system can be solved using standard conjugate gradient methods. 
Alternatively, the next proposition shows how the inverse $(\lambda A^\tT A + \frac{1}{\tau} I)^{-1}$ 
can be computed explicitly with the help of the Schur complement
and fast sine,-, cosine- and Fourier transforms.
The proposition refers to the one-dimensional spatial setting
but can be generalized to the three-dimensional case in a straightforward way using the results of Appendix C.
%
\begin{proposition} \label{lem:schur}
Let 
$
S_{N-1} \vcentcolon= \sqrt{\tfrac{2}{N}} \left( \sin \frac{jk\pi}{N} \right)_{j,k=1}^{N-1}
$
and
$
\dd^{\zero}_{N-1} \vcentcolon= \left( 4 \sin^2 \frac{k\pi}{2N}\right)_{k=1}^{N-1}
$.
Then the inverse of the matrix 
$
\lambda A^\tT A + \frac{1}{\tau} I
$
is given by
$$
\left(
\begin{array}{ll}
I    & - X^{-1}Y\\
0 & I
\end{array} 
\right)
\left(
\begin{array}{cc}
X^{-1}   & 0\\
0 & S^{-1}
\end{array} 
\right)
\left(
\begin{array}{cc}
I    & 0\\
-Y^\tT X^{-1}& I
\end{array} 
\right),
$$
where\\
{\rm i)} for mirror boundary conditions
\begin{align} \label{schur_mir}
Y &=  D_P^\tT \otimes D_N, \\
X^{-1} &=
I_P \otimes S_{N-1} \diag (\lambda N^2 \dd_{N-1}^{\zero} + \tfrac{1}{\tau} )^{-1} S_{N-1},\\
S^{-1} &= 
 (S_{P-1} \otimes C_N^\tT) 
\diag \Bigl(  \lambda P^2 \dd^{\zero}_{P-1} (1+ \tau \lambda N^2 \dd^{\mirr}_N )^{-1} + \tfrac{1}{\tau}  \Bigr) (S_{P-1} \otimes  C_N),
\end{align}
{\rm ii)} for periodic boundary conditions
\begin{align} \label{schur_per}
Y &=  D_P^\tT \otimes D_N^{\per}, \\
X^{-1} &=
I_P \otimes F_N \diag(\lambda N^2 \dd_N^{\per} + \tfrac{1}{\tau}) ^{-1} \bar F_N,\\
S^{-1}&  = 
(S_{P-1} \otimes F_N) \diag \Bigl(  \lambda P^2 \dd^{\zero}_{P-1} \otimes (1+ \tau \lambda N^2 \dd^{\per}_N )^{-1} + \tfrac{1}{\tau}  \Bigr)^{-1}(S_{P-1} \otimes \bar F_N).
\end{align}
\end{proposition}
The proof is given in Appendix C.
\subsection{Proximal Map of ${ J}_p$} \label{subsec:prox_J}
Step~2 of Algorithm~\ref{alg1} consists of an evaluation of the proximal map 
$\prox_{\frac{1}{\sigma} J_p}$ of $J_p$. 
This can be done using the proximal map $\prox_{ J^{\ast}_p}$ of $J^{\ast}_p$ and Moreau's identity $\prox_\phi (t) + \prox_{\phi^*}(t) = t$. 
Therefore, we state in the following first the dual function $J^{\ast}_p$.
%
\begin{lemma} \label{prox_J_p}
	For $p \in (1,+\infty)$ and $\frac1p + \frac1q = 1$ we have
	\begin{align}
	J_p^*(a,b) = 
	\left\{
	\begin{array}{ll}
	0&{\rm if} \; (a,b) \in {\cal K}_p,\\
	+\infty&{\rm otherwise}.
	\end{array}
	\right.
	\end{align}
	where
	$$
	{\cal K}_p \vcentcolon= \left\{ (a,b) \in \mathbb R^d \times \mathbb R\colon \tfrac{1}{q} |a|^q + b \le 0 \right\}.
	$$
\end{lemma}
%
For the proof we refer to \cite{Ji08} or \cite[Lemma 5.17]{San15}.
After this preparation we are now able to compute $\prox_{\frac{1}{\sigma} J_p}$.
%
\begin{proposition} \label{prox_j_p}
	Let $p \in (1,2]$ and $\frac1p + \frac1q = 1$.\\
	{\rm i)} Then for $x^* \in \mathbb R^d$, $y^* \in \mathbb R$
	and $\sigma > 0$ it holds
	\begin{align}
			\prox_{\frac{1}{\sigma} J_p}(x^*,y^*) 
			= 
			\left\{
			\begin{array}{ll}
			(0,0) & {\rm if} \;  (\sigma x^*,\sigma y^*) \in {\cal K}_p,\\
			\left( x^* \frac{h(\hat z)}{1+h(\hat z)} , y^* + \frac{1}{\sigma q} \hat z^q \right)& {\rm otherwise,}
			\end{array}
			\right.
	\end{align}	
	where
	$$
	h(z) \vcentcolon= (\sigma y^* + \tfrac{1}{q} z^q) z^{q-2}
	$$
	and
	$\hat z \in \mathbb R_{\ge 0}$ is the unique solution 
	of the equation
	\begin{equation} \label{poly_newton}
	z \left( 1+ h(z) \right) - \sigma | x^*| = 0
	\end{equation}
	in the interval $\big[\max(0, z_0)^{\tfrac1q},+\infty\big)$, where 
	$z_0\coloneqq - q \sigma y^*$.
	\\
	{\rm ii)} The Newton method converges for any starting point $z \ge z_0$
	quadratically to the largest zero of \eqref{poly_newton}.
\end{proposition}
%
\begin{proof}
i)	By Moreau's identity it holds $$\prox_\phi (t) + \prox_{\phi^*}(t) = t$$
	and since $(\tfrac{1}{\sigma} \phi)^*(t) = \tfrac{1}{\sigma} \phi^*(\sigma t)$
	we conclude
	\begin{align} \label{moreau}
	(\hat x, \hat y) &= \prox_{\frac{1}{\sigma} J_p}(x^*,y^*)\\& = (x^*,y^*) - \frac{1}{\sigma} \prox_{\sigma J_p^*}(\sigma x^*,\sigma y^*),
	\end{align}
	where $\prox_{\frac{1}{\sigma} J_p^*} = \prox_{J_p^*}$ since $J_p^*$ is an indicator function.
	Note that by definition of $J_p$ we have that $\hat y \ge 0$ and $\hat y = 0$ only if $|\hat x| = 0$.	
	Now, $\prox_{J_p^*}(\sigma x^*,\sigma y^*)$ is the orthogonal projection of
	$(\sigma x^*,\sigma y^*)$ onto the set ${\cal K}_p$ (we could also compute the epigraphical projection of 
	$(\sigma x^*,\sigma y^*)$
	onto the epigraph of $\phi(x) = \tfrac1q |x|^q$ and reflect $y$, see also Figure~\ref{fig:epi}).	
	\begin{figure}	
		\centering
		\includegraphics[width=0.4\textwidth]{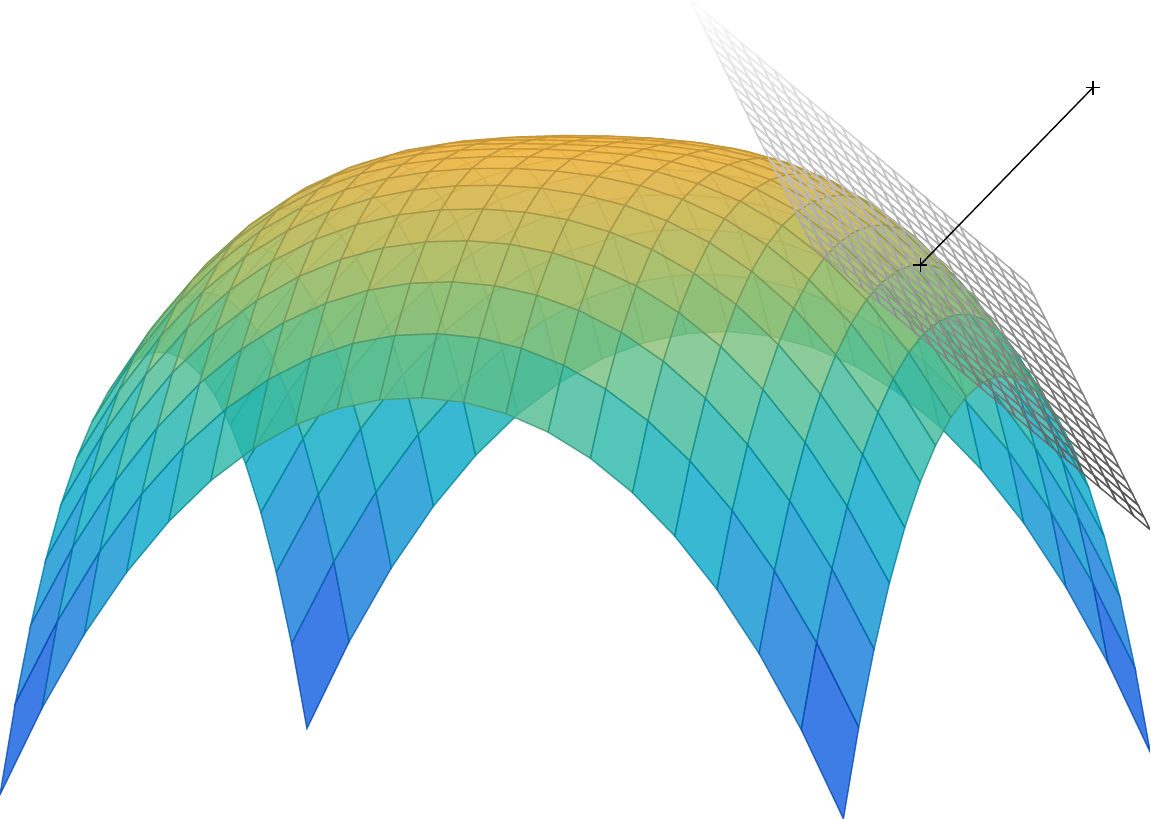}
		\caption{Projection onto the graph of the function $\phi(x) = -\tfrac1q 
		|x|^q$ for $q=3$.}	\label{fig:epi}			
	\end{figure}	
	If $(\sigma x^*, \sigma y^*) \in {\cal K}_p$, then 
	$\prox_{J_p^*}(\sigma x^*,\sigma y^*) = (\sigma x^*,\sigma y^*)$ 
	and 
	$(\hat x, \hat y) = (0,0)$.
	So let $(\sigma x^*, \sigma y^*) \not \in {\cal K}_p$, that means
	\begin{align} \label{not_K}
	\sigma y^* + \tfrac1q |\sigma x^{\ast}|^q >0.
	\end{align}
	The tangent plane of the boundary of 
	${\cal K}_p$ in $(x,y) = (x, -\tfrac1q |x|^q)$ 
	is spanned by
	the vectors $(e_i^\tT, - |x|^{q-2} x_i)^\tT$, $i=1,\ldots,d$, 
	where $e_i \in \mathbb R^d$ denotes the $i$-th canonical unit vector.
	Hence, the projection $(x,y)$ is determined by $y = -\tfrac1q |x|^q$ and
	\begin{align}
	0 & = \left\langle 
	\begin{pmatrix}  \sigma x^* \\ \sigma y^* \end{pmatrix}
	-
	\begin{pmatrix}   x \\  y\end{pmatrix},
	\begin{pmatrix}  e_i \\ - |x|^{q-2} x_i \end{pmatrix}
	\right\rangle\\
	& = \sigma x^*_i - x_i - (\sigma y^* - y) |x|^{q-2} x_i, \qquad i=1,\ldots,d
	\end{align}
	so that
	\begin{align}
	x_i = \frac{\sigma x^*_i}{1+h(|x|)}, \quad i=1,\ldots,d.
	\end{align}
	Summing over the squares of the last equations gives
	\begin{align}
	|x|^2 = |x^*|^2 \frac{\sigma^2}{\left(1+h(|x|)\right)^2}.
	\end{align}
	Since a solution has to fulfill 
	$$
	\hat y = y^* -\tfrac{1}{\sigma} y = y^* + \tfrac{1}{q \sigma } |x|^q = \sigma h(|x|) |x|^{2-q} > 0,
	$$
	it remains to search for the solutions with $h(|x|) > 0$.
	Now, $h(z) >0$ is fulfilled for $z >0$ if and only if
	\begin{align} \label{z0}
	\sigma y^* + \tfrac1q z^q > 0,
	\end{align}
	which is the case if and only if $z > z_0 \vcentcolon= \max(0, - q \sigma y^*)^{\tfrac1q}$.
	Then $z \vcentcolon= |x|$ has to satisfy the equation
	\begin{align} \label{newton}
	z \left( 1+h(z) \right) =  \sigma|x^*|.
	\end{align}
	The function 
	$$
	\varphi(z) \vcentcolon= z \left(1+h(z) \right) - \sigma|x^*|
	$$
	has exactly one zero in $[z_0,+\infty)$. Indeed, by definition of $z_0$ and~\eqref{z0} we see that
	$
	\varphi(z_0) \le 0
	$, but on the other hand we have $\varphi(z) \rightarrow +\infty$ as $z \rightarrow +\infty$, so that $\varphi$ has at least a zero in $[z_0,+\infty)$.
	Since $q \ge 2$ for $p \le 2$ we have for
	$$
	h'(z) = z^{2q-3} + (q-2) (\sigma y^* + \frac{1}{q} z^q) z^{q-3} > 0, \quad z > z_0.
	$$
	Hence $h$ and then also $\varphi$ is strictly monotone increasing for $z > z_0$.
	Therefore $\varphi$ has at most one zero in $[z_0,+\infty)$.
	Finally, the assertion follows by plugging in $(x,y)$ in \eqref{moreau}.
	\\[1ex]
ii)	Straightforward computation gives
	\begin{align}
	h''(z) &= (3q-5)z^{2q-4} + (q-2)(q-3)( \sigma y^* + \tfrac1q z^q ) z^{q-4},\\
	\varphi'(z) &= 1 + h(z) + zh'(z),\\
	\varphi''(z) &= 2h'(z) + z h''(z)\\
	& = 3(q-1) z^{2q-3} + (q-1)(q-2)(\sigma y^* + \tfrac1q z^q) z^{q-3} >0, \quad z>z_0.
	\end{align}
	Since $\varphi$ is monotone increasing and strictly convex
	for $z \ge z_0$, the Newton method converges for any starting point $z \ge z_0$
	quadratically.
	\end{proof}
%
\section{Numerical Results} \label{sec:numerics}
In the following we provide several numerical examples. In all cases we used 
$P=32$ time steps and $2000$ iterations in Algorithm~\ref{alg1} and~\ref{alg2}, respectively. The parameters 
$\sigma$ and $\tau$ were set to $\sigma = 50$ and $\tau = \frac{0.99}{\sigma}$, 
so that $\sigma\tau<1$, which guarantees the convergence of the algorithms. Of course, the algorithms do not use tensor products, but relations such as stated in~\eqref{vec_mat} and their higher dimensional versions. 
The algorithms were implemented in \textsc{Matlab} and the computations were performed on a Dell computer with an Intel Core i7,
2.93\,Ghz and 8\,GB of RAM using \textsc{Matlab} 2014, Version 2014b
on Ubuntu 14.04 LTS. 
Exemplary, the run time for $100\times 100$ color images and 32 time steps varies between 10 and 15 minutes for the constrained and the penalized method, depending on the parameter choices for $p,\sigma,\tau$ and $\lambda$.
In our current implementation of the penalized method the  fast transform approach is nearly as time consuming as
the iterative solution of the linear system of equations with the CG method and an adequate initialization.
If not explicitly stated otherwise, the results are  displayed for $p=2$, in which case the computation of the zeros in~\eqref{poly_newton} slightly simplifies. \\
With our first experiments we illustrate the difference between mirror and periodic boundary conditions in the color dimension, where  at this point that the initial and the final images have the same mass. 
The images are displayed at intermediate timepoints $t = \frac{i}{8},$ where $i = 0,\dots, 8$.
In Figure~\ref{Fig:boundary_color}, the transport of a red Gaussian into a blue one is shown, either with mirror or with periodic boundary conditions in the color dimension. Figure~\ref{Fig:boundary_echt}\footnote{Images from Wikimedia Commons: 
AGOModra\_aurora.jpg by 
Comenius University under CC BY SA 3.0, 
Aurora-borealis\_andoya.jpg by M.~Buschmann under CC BY 3.0.} depicts the transport between two real images of polar lights.
In order to have the equal mass constraint fulfilled, we first normalized both images to mass 1 and afterwards multiplied them with a common factor such that both images have realistic colors. Of course, this procedure works only if the initial and the final image share approximately the same mass.
  In both cases the use of periodic boundary conditions yields more realistic results.\\%
\begin{figure*} 
	\centering
	{\includegraphics[width=.1\textwidth]{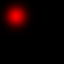}} 
	{\includegraphics[width=.1\textwidth]{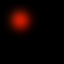}} 
	{\includegraphics[width=.1\textwidth]{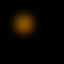}} 
	{\includegraphics[width=.1\textwidth]{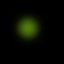}} 
	{\includegraphics[width=.1\textwidth]{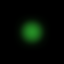}} 
	{\includegraphics[width=.1\textwidth]{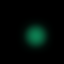}} 
	{\includegraphics[width=.1\textwidth]{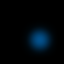}} 
	{\includegraphics[width=.1\textwidth]{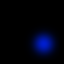}} 
	{\includegraphics[width=.1\textwidth]{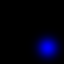}} 
	\\[1.5ex]
	{\includegraphics[width=.1\textwidth]{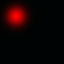}} 
	{\includegraphics[width=.1\textwidth]{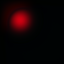}} 
	{\includegraphics[width=.1\textwidth]{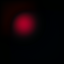}} 
	{\includegraphics[width=.1\textwidth]{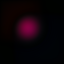}} 
	{\includegraphics[width=.1\textwidth]{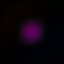}} 
	{\includegraphics[width=.1\textwidth]{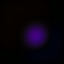}} 
	{\includegraphics[width=.1\textwidth]{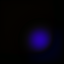}} 
	{\includegraphics[width=.1\textwidth]{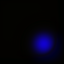}} 
	{\includegraphics[width=.1\textwidth]{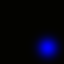}}
	\caption{\label{fig:boundary}
		Dynamic optimal transport between a red Gaussian and a blue one by the constrained model \protect\eqref{model_disc} with different boundary conditions
		for the third (RGB) dimension. The initial and final images  have the same mass.
		Top: mirror boundary conditions, bottom: periodic boundary conditions. 
			}\label{Fig:boundary_color}
\end{figure*}
\begin{figure*} \centering
	{\includegraphics[width=.1\textwidth]{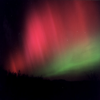}} 
	{\includegraphics[width=.1\textwidth]{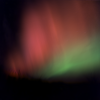}} 
	{\includegraphics[width=.1\textwidth]{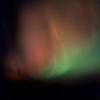}} 
	{\includegraphics[width=.1\textwidth]{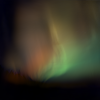}} 
	{\includegraphics[width=.1\textwidth]{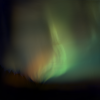}} 
	{\includegraphics[width=.1\textwidth]{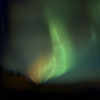}} 
	{\includegraphics[width=.1\textwidth]{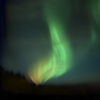}} 
	{\includegraphics[width=.1\textwidth]{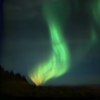}} 
	{\includegraphics[width=.1\textwidth]{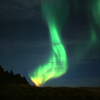}} \\[1.5ex]
		{\includegraphics[width=.1\textwidth]{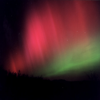}} 
		{\includegraphics[width=.1\textwidth]{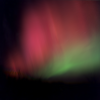}} 
		{\includegraphics[width=.1\textwidth]{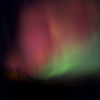}} 
		{\includegraphics[width=.1\textwidth]{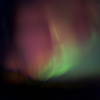}} 
		{\includegraphics[width=.1\textwidth]{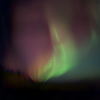}} 
		{\includegraphics[width=.1\textwidth]{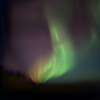}} 
		{\includegraphics[width=.1\textwidth]{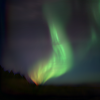}} 
		{\includegraphics[width=.1\textwidth]{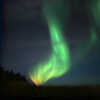}} 
		{\includegraphics[width=.1\textwidth]{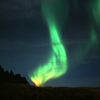}}
	\caption{Dynamic optimal transport between two polar lights by the constrained model \protect\eqref{model_disc} with different boundary conditions
		for the third (RGB) dimension. The initial and final images  have the same mass.
		Top: mirror boundary conditions, bottom: periodic boundary conditions. 
				}\label{Fig:boundary_echt}
\end{figure*}%
{ Further examples of the constrained model~\eqref{model_disc} for several Gaussians and real images are given in 
Figures~\ref{Fig:intensity_transitions} and~\ref{Fig:RGB_transport}\footnote{Images from Wikimedia Commons: 				
				Europe\_satellite\_orthographic.jpg and Earthlights\_2002.jpg 
				by 
				NASA,
				K\"ohlbrandbr\"ucke5478.jpg by G.~Ries under CC BY SA 2.5, 
				K\"ohlbrandbr\"ucke.jpg by HafenCity1 under CC BY 3.0.}.
The first row in Figures~\ref{Fig:intensity_transitions} shows the transport of a red and a yellow Gaussian into a cyan and a blue Gaussian. The red and the blue Gaussian are spatially more extended compared to the yellow and the cyan one, but due to the fact that yellow and cyan have higher intensities, the masses of the red and blue Gaussians are approximately the same those of the yellow respective the cyan ones. 
This results in a very low interaction between the Gaussians during the transport, which is sightly visible in the background. 
Mainly, the red Gaussian is transported via a light red into cyan and the yellow Gaussian is transported over violet to blue. 	
The next row displays the intensity of the transported color images $\frac{1}{3}(R+G+B)$, in contrast to the (two-dimensional) transported intensity images in the third row. 
One sees  slight differences which arise due to the fact that the small mass difference can be transported only spatially and not through the color channels.\\
		The experiment is repeated in the fourth until sixth row, but this time the yellow and the cyan Gaussian are spatially more extended, thus having a significantly higher mass compared to the red respective the blue Gaussian. As a consequence, the interaction during the transport is higher, which is also clearly visible in the corresponding intensity images (fifth row). The color of the red Gaussian is transported similar as before, while the yellow Gaussian splits into two parts, one of them changing (as before) over violet to blue, while the other one goes over a light green to cyan. Further, as the Gaussians do not only travel in space, but also in color direction, the results are slightly smoother compared to the two-dimensional intensity transport, shown in the last row. }
			 \\					
			\begin{figure*}	\centering	{\includegraphics[width=.10\textwidth]{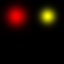}} 
				{\includegraphics[width=.10\textwidth]{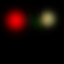}} 
				{\includegraphics[width=.10\textwidth]{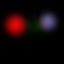}} 
				{\includegraphics[width=.10\textwidth]{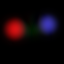}} 
				{\includegraphics[width=.10\textwidth]{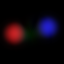}} 
				{\includegraphics[width=.10\textwidth]{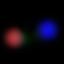}} 
				{\includegraphics[width=.10\textwidth]{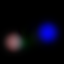}} 
				{\includegraphics[width=.10\textwidth]{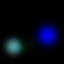}} 
				{\includegraphics[width=.10\textwidth]{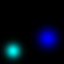}}				\\[1.5ex] 	
				{\includegraphics[width=.10\textwidth]{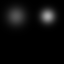}} 
				{\includegraphics[width=.10\textwidth]{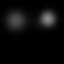}} 
				{\includegraphics[width=.10\textwidth]{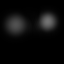}} 
				{\includegraphics[width=.10\textwidth]{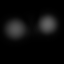}} 
				{\includegraphics[width=.10\textwidth]{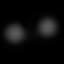}} 
				{\includegraphics[width=.10\textwidth]{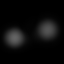}} 
				{\includegraphics[width=.10\textwidth]{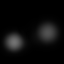}} 
				{\includegraphics[width=.10\textwidth]{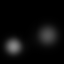}} 
				{\includegraphics[width=.10\textwidth]{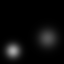}}	\\[1.5ex] 
				{\includegraphics[width=.10\textwidth]{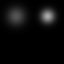}} 
				{\includegraphics[width=.10\textwidth]{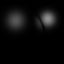}} 
				{\includegraphics[width=.10\textwidth]{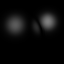}} 
				{\includegraphics[width=.10\textwidth]{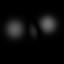}} 
				{\includegraphics[width=.10\textwidth]{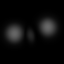}} 
				{\includegraphics[width=.10\textwidth]{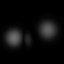}} 
				{\includegraphics[width=.10\textwidth]{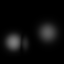}} 
				{\includegraphics[width=.10\textwidth]{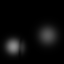}} 
				{\includegraphics[width=.10\textwidth]{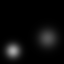}}	\\[1.5ex] 
				{\includegraphics[width=.10\textwidth]{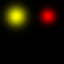}} 
				{\includegraphics[width=.10\textwidth]{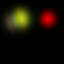}} 
				{\includegraphics[width=.10\textwidth]{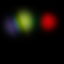}} 
				{\includegraphics[width=.10\textwidth]{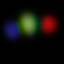}} 
				{\includegraphics[width=.10\textwidth]{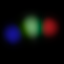}} 
				{\includegraphics[width=.10\textwidth]{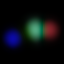}} 
				{\includegraphics[width=.10\textwidth]{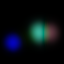}} 
				{\includegraphics[width=.10\textwidth]{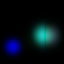}} 
				{\includegraphics[width=.10\textwidth]{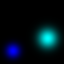}}		 \\[1.5ex] 
				{\includegraphics[width=.10\textwidth]{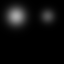}} 
				{\includegraphics[width=.10\textwidth]{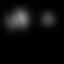}} 
				{\includegraphics[width=.10\textwidth]{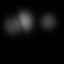}} 
				{\includegraphics[width=.10\textwidth]{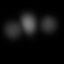}} 
				{\includegraphics[width=.10\textwidth]{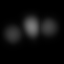}} 
				{\includegraphics[width=.10\textwidth]{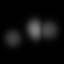}} 
				{\includegraphics[width=.10\textwidth]{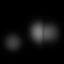}} 
				{\includegraphics[width=.10\textwidth]{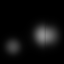}} 
				{\includegraphics[width=.10\textwidth]{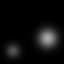}} \\[1.5ex] 	
				{\includegraphics[width=.10\textwidth]{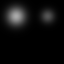}} 
				{\includegraphics[width=.10\textwidth]{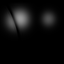}} 
				{\includegraphics[width=.10\textwidth]{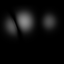}} 
				{\includegraphics[width=.10\textwidth]{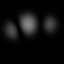}} 
				{\includegraphics[width=.10\textwidth]{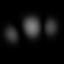}} 
				{\includegraphics[width=.10\textwidth]{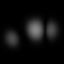}} 
				{\includegraphics[width=.10\textwidth]{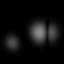}} 
				{\includegraphics[width=.10\textwidth]{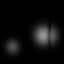}} 
				{\includegraphics[width=.10\textwidth]{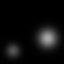}}	\\[1.5ex] 			
				\caption{Example for different color transitions obtained with the constrained model \protect\eqref{model_disc} 
				and periodic boundary conditions (first and fourth row), where the initial and final color images have the same mass. The second and fifth row show the corresponding intensity images, while the third and sixth row give the results obtained using two-dimensional transport of the initial and the final intensity images.}\label{Fig:intensity_transitions}
			\end{figure*}			
Also for the real images, we assume that the images have the same mass.
Indeed, the initial and final images had approximately the same overall sum of values, 
so that our normalization had no significant effect.  
 In the first row of Figure~\ref{Fig:RGB_transport}, 
a topographic map of Europe is transported into a satellite image of Europe at night. 
The second row displays the transport between two images of the 
K\"ohlbrandbr\"ucke in Hamburg. 
In both cases one nicely sees a continuous change of color and shape during the transport.\\
\begin{figure*} \centering
{\includegraphics[width=.10\textwidth]{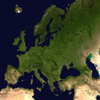}} 
{\includegraphics[width=.10\textwidth]{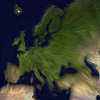}} 
{\includegraphics[width=.10\textwidth]{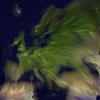}} 
{\includegraphics[width=.10\textwidth]{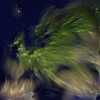}} 
{\includegraphics[width=.10\textwidth]{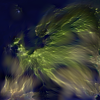}} 
{\includegraphics[width=.10\textwidth]{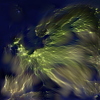}} 
{\includegraphics[width=.10\textwidth]{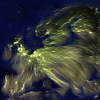}} 
{\includegraphics[width=.10\textwidth]{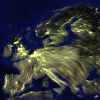}} 
{\includegraphics[width=.10\textwidth]{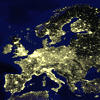}} \\[1.5ex]
{\includegraphics[width=.10\textwidth]{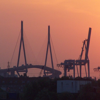}} 
{\includegraphics[width=.10\textwidth]{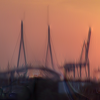}} 
{\includegraphics[width=.10\textwidth]{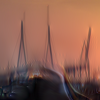}} 
{\includegraphics[width=.10\textwidth]{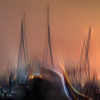}} 
{\includegraphics[width=.10\textwidth]{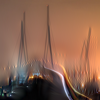}} 
{\includegraphics[width=.10\textwidth]{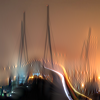}} 
{\includegraphics[width=.10\textwidth]{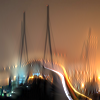}} 
{\includegraphics[width=.10\textwidth]{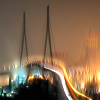}} 
{\includegraphics[width=.10\textwidth]{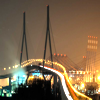}} 
\caption{\label{Fig:RGB_transport}
	Dynamic optimal transport between RGB images by the constrained model \protect\eqref{model_disc} with periodic boundary conditions. 
	The initial and final images  have the same mass.
	 }
\end{figure*}
In Figure~\ref{Fig:color_transitions} we give further examples for the transport of several Gaussians which may have different shapes 
in order to illustrate the transport of color and shape. Here, the initial and final images have different masses. {The first row shows the transport of a yellow and a red Gaussian placed at the top of the image into a green and a blue Gaussian placed at the bottom. At this point, the yellow and the blue Gaussian are slightly more spatially extended compared to the red respective the green one. The red Gaussian changes over violet to blue. The yellow Gaussian, however, splits into two parts. While the main part is transported to green, a small part is separated and changes over orange to blue.
In the second row, again a red and a yellow Gaussian are transported from the top into a green and a blue Gaussian at the bottom, but this time the yellow and the green Gaussian are spatially more extended. Additionally, the Gaussians are no longer isotropic but have an ellipsoidal shape. In this case, additionally to the color also the shape changes continuously during time. 
	Finally, the third row displays the transport of a white Gaussian into a yellow one. It shows that there appear no artificial colors during the transport, but the color transition proceeds as one may expect when looking at the RGB cube.   \\
\begin{figure*}	\centering
	{\includegraphics[width=.10\textwidth]{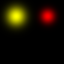}} 
	{\includegraphics[width=.10\textwidth]{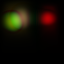}} 
	{\includegraphics[width=.10\textwidth]{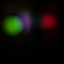}} 
	{\includegraphics[width=.10\textwidth]{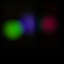}} 
	{\includegraphics[width=.10\textwidth]{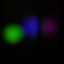}} 
	{\includegraphics[width=.10\textwidth]{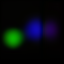}} 
	{\includegraphics[width=.10\textwidth]{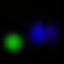}} 
	{\includegraphics[width=.10\textwidth]{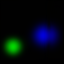}} 
	{\includegraphics[width=.10\textwidth]{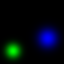}}				\\[1.5ex]
	{\includegraphics[width=.10\textwidth]{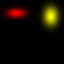}} 
	{\includegraphics[width=.10\textwidth]{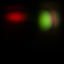}} 
	{\includegraphics[width=.10\textwidth]{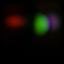}} 
	{\includegraphics[width=.10\textwidth]{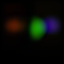}} 
	{\includegraphics[width=.10\textwidth]{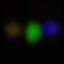}} 
	{\includegraphics[width=.10\textwidth]{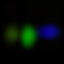}} 
	{\includegraphics[width=.10\textwidth]{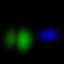}} 
	{\includegraphics[width=.10\textwidth]{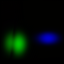}} 
	{\includegraphics[width=.10\textwidth]{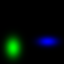}}\\[1.5ex]
	{\includegraphics[width=.10\textwidth]{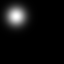}} 
	{\includegraphics[width=.10\textwidth]{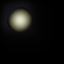}} 
	{\includegraphics[width=.10\textwidth]{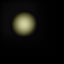}} 
	{\includegraphics[width=.10\textwidth]{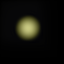}} 
	{\includegraphics[width=.10\textwidth]{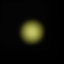}} 
	{\includegraphics[width=.10\textwidth]{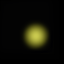}} 
	{\includegraphics[width=.10\textwidth]{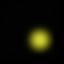}} 
	{\includegraphics[width=.10\textwidth]{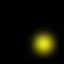}} 				
	{\includegraphics[width=.10\textwidth]{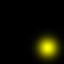}}		
	\caption{Example for different color and shape transitions by the constrained model \protect\eqref{model_disc} with periodic boundary conditions.
		The initial and final images do not have the same mass.}\label{Fig:color_transitions}
\end{figure*}
Next we compare our approach with the approach of Rabin et al.~\cite{RPDB12} for microtextures,
 which is to the best of our knowledge the only approach that extends the dynamic optimal transport problem to a  special class of color images. 
Note, however, that their approach is completely different from ours and works only for microtextures. 
At this point, microtextures are textures that fulfill the assumption of being robust
towards phase randomization, in contrast to macrotextures, which usually contain periodic patterns with big visible elements (such as
brick walls) or - more generally - the elements that form the texture-pattern are spatially
arranged, see e.g.~\cite{GGM11}. Based on the fact  that microtextures can be well modeled as multivariate Gaussian distributions 
the authors of~\cite{RPDB12} propose to compute geodesics with respect to the Wasserstein distance $W_2$ between the Gaussian distributions that are estimated from the input textures $f_0$ and $f_1$. This approach has the advantage that there exist closed-form solutions for the dynamic optimal transport between Gaussian measures. However, it is limited to the special class of microtextures, as natural images are not robust towards a randomization of their Fourier phase. In Figure~\ref{Fig:Microtexturemodel} 
we compare the results of our approach  with the one for microtextures. In the case of microtextures both approaches yield similar results. 
Note that the approach of Rabin et al.~\cite{RPDB12} may fail for images which are orthogonal at some frequencies in Fourier domain.  
The second example demonstrates  that the microtexture technique~\cite{RPDB12}  fails for natural images which possess contours and edges. \\
\begin{figure*} \centering
	{\includegraphics[width=.10\textwidth]{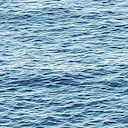}} 
	{\includegraphics[width=.10\textwidth]{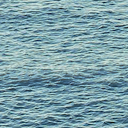}} 
	{\includegraphics[width=.10\textwidth]{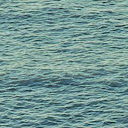}} 
	{\includegraphics[width=.10\textwidth]{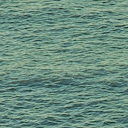}} 
	{\includegraphics[width=.10\textwidth]{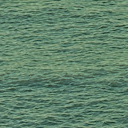}} 
	{\includegraphics[width=.10\textwidth]{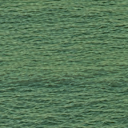}} 
	{\includegraphics[width=.10\textwidth]{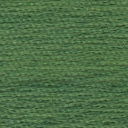}} 
	{\includegraphics[width=.10\textwidth]{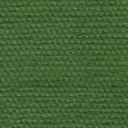}} 
	{\includegraphics[width=.10\textwidth]{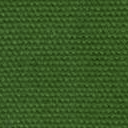}} \\[1.5ex]	
	{\includegraphics[width=.10\textwidth]{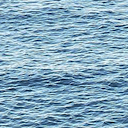}} 
	{\includegraphics[width=.10\textwidth]{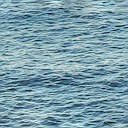}} 
	{\includegraphics[width=.10\textwidth]{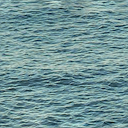}} 
	{\includegraphics[width=.10\textwidth]{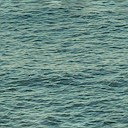}} 
	{\includegraphics[width=.10\textwidth]{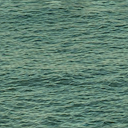}} 
	{\includegraphics[width=.10\textwidth]{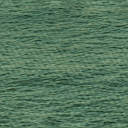}} 
	{\includegraphics[width=.10\textwidth]{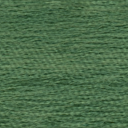}} 
	{\includegraphics[width=.10\textwidth]{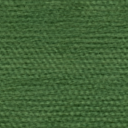}} 
	{\includegraphics[width=.10\textwidth]{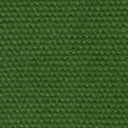}} 
 \\[1.5ex]
	{\includegraphics[width=.1\textwidth]{images/Polar_1.png}} 
	{\includegraphics[width=.1\textwidth]{images/Polar_5.png}} 
	{\includegraphics[width=.1\textwidth]{images/Polar_9.png}} 
	{\includegraphics[width=.1\textwidth]{images/Polar_13.png}} 
	{\includegraphics[width=.1\textwidth]{images/Polar_17.png}} 
	{\includegraphics[width=.1\textwidth]{images/Polar_21.png}} 
	{\includegraphics[width=.1\textwidth]{images/Polar_25.png}} 
	{\includegraphics[width=.1\textwidth]{images/Polar_29.png}} 
	{\includegraphics[width=.1\textwidth]{images/Polar_33.png}}  \\[1.5ex]
		{\includegraphics[width=.10\textwidth]{images/Polar_1.png}} 
		{\includegraphics[width=.10\textwidth]{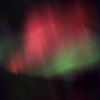}} 
		{\includegraphics[width=.10\textwidth]{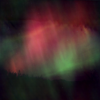}} 
		{\includegraphics[width=.10\textwidth]{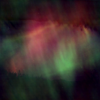}} 
		{\includegraphics[width=.10\textwidth]{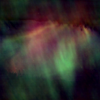}} 
		{\includegraphics[width=.10\textwidth]{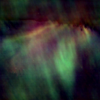}} 
		{\includegraphics[width=.10\textwidth]{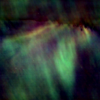}} 
		{\includegraphics[width=.10\textwidth]{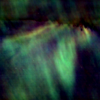}} 
		{\includegraphics[width=.10\textwidth]{images/Polar_33.png}}
	\caption{\label{Fig:Microtexturemodel}
	Comparison of our constrained model \protect\eqref{model_disc} with periodic boundary conditions with the
	approach in ~\cite{RPDB12} for microtextures and the polar lights.
	In both cases, our	RGB model is on top of the series and microtexture model is at the bottom. }
\end{figure*}
Next, we turn to the penalized model~\eqref{model_disc_pen}. 
Figure~\ref{Fig:lambda} shows the influence of 
the regularization parameter $\lambda$ when transporting a red Gaussian into a yellow one. Here, the initial and the final image have 
significantly different mass. The images are 
	displayed at intermediate timepoints $t = \frac{i}{8}, i = 0,\dots, 
	8$.
The results change for increasing $\lambda$ 
from a   nearly linear interpolation of the images to a transport of the mass. 
Further, for large $\lambda$ the results approach the one obtained with the constrained model~\eqref{model_disc}, which is reasonable.\\
\begin{figure*} \centering
	{\includegraphics[width=.10\textwidth]{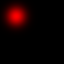}} 
	{\includegraphics[width=.10\textwidth]{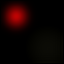}} 
	{\includegraphics[width=.10\textwidth]{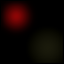}} 
	{\includegraphics[width=.10\textwidth]{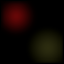}}
	{\includegraphics[width=.10\textwidth]{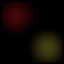}}
	{\includegraphics[width=.10\textwidth]{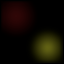}} 
	{\includegraphics[width=.10\textwidth]{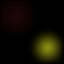}}  
	{\includegraphics[width=.10\textwidth]{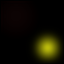}}  
	{\includegraphics[width=.10\textwidth]{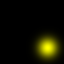}} \\[1.5ex]
	{\includegraphics[width=.10\textwidth]{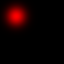}} 
	{\includegraphics[width=.10\textwidth]{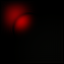}} 
	{\includegraphics[width=.10\textwidth]{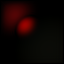}} 
	{\includegraphics[width=.10\textwidth]{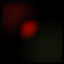}}
	{\includegraphics[width=.10\textwidth]{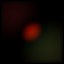}}
	{\includegraphics[width=.10\textwidth]{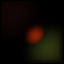}} 
	{\includegraphics[width=.10\textwidth]{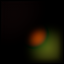}}  
	{\includegraphics[width=.10\textwidth]{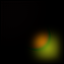}}  
	{\includegraphics[width=.10\textwidth]{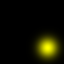}} \\[1.5ex]
	{\includegraphics[width=.10\textwidth]{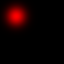}} 
	{\includegraphics[width=.10\textwidth]{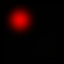}} 
	{\includegraphics[width=.10\textwidth]{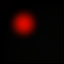}} 
	{\includegraphics[width=.10\textwidth]{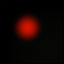}}
	{\includegraphics[width=.10\textwidth]{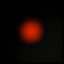}}
	{\includegraphics[width=.10\textwidth]{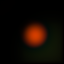}} 
	{\includegraphics[width=.10\textwidth]{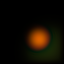}}  
	{\includegraphics[width=.10\textwidth]{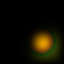}}  
	{\includegraphics[width=.10\textwidth]{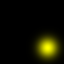}} \\[1.5ex]  
	{\includegraphics[width=.10\textwidth]{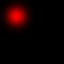}} 
	{\includegraphics[width=.10\textwidth]{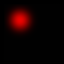}} 
	{\includegraphics[width=.10\textwidth]{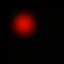}} 
	{\includegraphics[width=.10\textwidth]{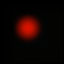}}
	{\includegraphics[width=.10\textwidth]{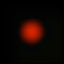}}
	{\includegraphics[width=.10\textwidth]{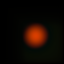}} 
	{\includegraphics[width=.10\textwidth]{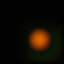}}  
	{\includegraphics[width=.10\textwidth]{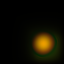}}  
	{\includegraphics[width=.10\textwidth]{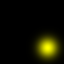}}  
	\\[1.5ex]  
	{\includegraphics[width=.10\textwidth]{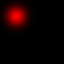}} 
	{\includegraphics[width=.10\textwidth]{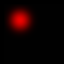}} 
	{\includegraphics[width=.10\textwidth]{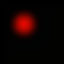}} 
	{\includegraphics[width=.10\textwidth]{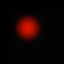}}
	{\includegraphics[width=.10\textwidth]{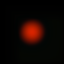}}
	{\includegraphics[width=.10\textwidth]{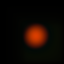}} 
	{\includegraphics[width=.10\textwidth]{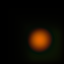}}  
	{\includegraphics[width=.10\textwidth]{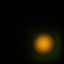}}  
	{\includegraphics[width=.10\textwidth]{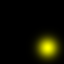}}  	
	\caption{Comparison of penalized and constrained color optimal transport (from top to bottom): 
	penalized optimal transport for different regularization parameters $\lambda\in \{0.1,1,10,100\}$ and constrained optimal transport. 
	}\label{Fig:lambda}
\end{figure*}
Finally, we consider the influence of the parameter $p \in (1,2]$. 
The corresponding results for our penalized mo\-del~\eqref{model_disc_pen} are given in Figures~\ref{Fig:color_echt2} and~\ref{Fig:diffP}. 
In all our experiments we observed only rather small differences.
Note however that in~\cite{COO14} Wasserstein barycenters were considered
for $p=1,2,3$ which show significant differences.\\
\begin{figure*}\centering	
	{\includegraphics[width=.19\textwidth]{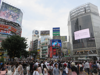}} 
	{\includegraphics[width=.19\textwidth]{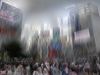}} 
	{\includegraphics[width=.19\textwidth]{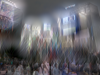}} 
	{\includegraphics[width=.19\textwidth]{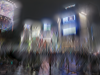}} 
	{\includegraphics[width=.19\textwidth]{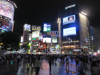}} 
	\\[1.5ex] 	
			{\includegraphics[width=.19\textwidth]{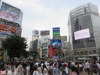}} 
			{\includegraphics[width=.19\textwidth]{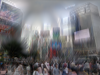}} 
			{\includegraphics[width=.19\textwidth]{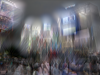}} 
			{\includegraphics[width=.19\textwidth]{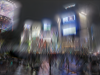}} 
			{\includegraphics[width=.19\textwidth]{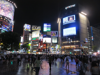}} 
\caption{Dynamic optimal transport of RGB images using the penalized model \protect\eqref{model_disc_pen} with periodic boundary conditions.
Comparison of $p=2$ (top) and $p=1.5$ (bottom) for $\lambda = 1$.
 The images are displayed at intermediate timepoints $t = \frac{i}{4}, i = 0,\dots, 4$.}
\label{Fig:color_echt2} 
\end{figure*}
\begin{figure*}
	 \centering
	{\includegraphics[width=.18\textwidth]{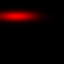}}
	{\includegraphics[width=.18\textwidth]{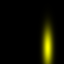}}
	{\includegraphics[width=.18\textwidth]{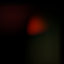}}
	{\includegraphics[width=.18\textwidth]{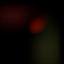}}
	{\includegraphics[width=.18\textwidth]{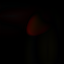}}
	\caption{\label{Fig:diffP}
		Dynamic optimal transport of RGB images using the penalized model \protect\eqref{model_disc_pen} with periodic boundary conditions and
		$\lambda=1$ for $p=2$ and 
		$p=1.5$. From left to right: Initial images $f_0$, $f_1$, result for 
		$p=2$ and $p=1.5$ at time $t=0.5$ and the absolute difference between 
		the two results.}
\end{figure*}
Further examples and videos, in particular for real images, 
can be found on our website { \small
  \url{http://www.mathematik.uni-kl.de/imagepro/members/laus/color-OT}}.

\section{Summary and Conclusions} \label{sec:conclusions}
Our contribution can be summarized as follows:
\begin{itemize}
	\item[i)]  
	We propose  two discrete  variational models for the interpolation of RGB color images 
	based on the dynamic optimal transport approach. To this end, we consider color images as
	three-dimensional objects, where the ``RGB direction'' is handled in a periodic way.
	We focus on a discrete matrix-vector approach.
	\item[ii)] Our first model relaxes the continuity constraint so that a transport between images of different mass
	is possible, while the second model allows even more flexibility by just penalizing the continuity constraint
	with different regularization parameters.
	\item[iii)] We provided an existence proof and a brief discussion on the uniqueness of the minimizer.
	\item[iv)]  Interestingly, the step in the chosen primal-dual algorithm which takes the continuity constraint into account requires the solution
	of four-dimensional Poisson equations with simultaneous  mirror/periodic (constrained model) or zero/mirror/periodic boundary conditions (penalized model). 
	Here, fast sine, cosine and Fourier transforms come into the play.
	\item[v)] 
	We consider the case $p\in (1,2]$ and
	give a careful analysis of the proximal mapping of $J_p$, $p \in (1,2]$. 
	This includes the determination of a starting point for the
	Newton algorithm to ensure its quadratic convergence and  a stable performance of the overall algorithm.
	\item[vi)] We show numerous numerical examples. 
\end{itemize}
There are several directions for future work.
\begin{enumerate}
\item[1)]
One possibility is to add several additional priors.
 So far, the present model has difficulties to transfer sharp contours.
A remedy could be the penalization of a total variation (TV) term with respect to $f$, which results for some $\gamma>0$ in the functional
\begin{equation}
\argmin_{ (m,f) \in {\cal C}}  \big\{ \| J_p(S_{\rm m} m, S_{\rm f} f + f^+) \|_1 + \gamma {\rm TV}(f) \big\}.\label{TV_functional}
\end{equation}
In the following we use a spatial TV term, summed over time, i.e., in one dimension
\begin{align}
	{\rm TV}(f) = \| (I_P \otimes D_N^\tT) f \|_1.
\end{align}
Figure~\ref{Fig:TVbump} shows the performance of such a model in the one-dimensional case.
Here, the approach without the TV term leads to some overshooting and blurring at the edges.
With the TV term the transport takes place in a more natural way, in particular the sharp edges are preserved.
\begin{figure}
\centering
\includegraphics[width=.45\textwidth]{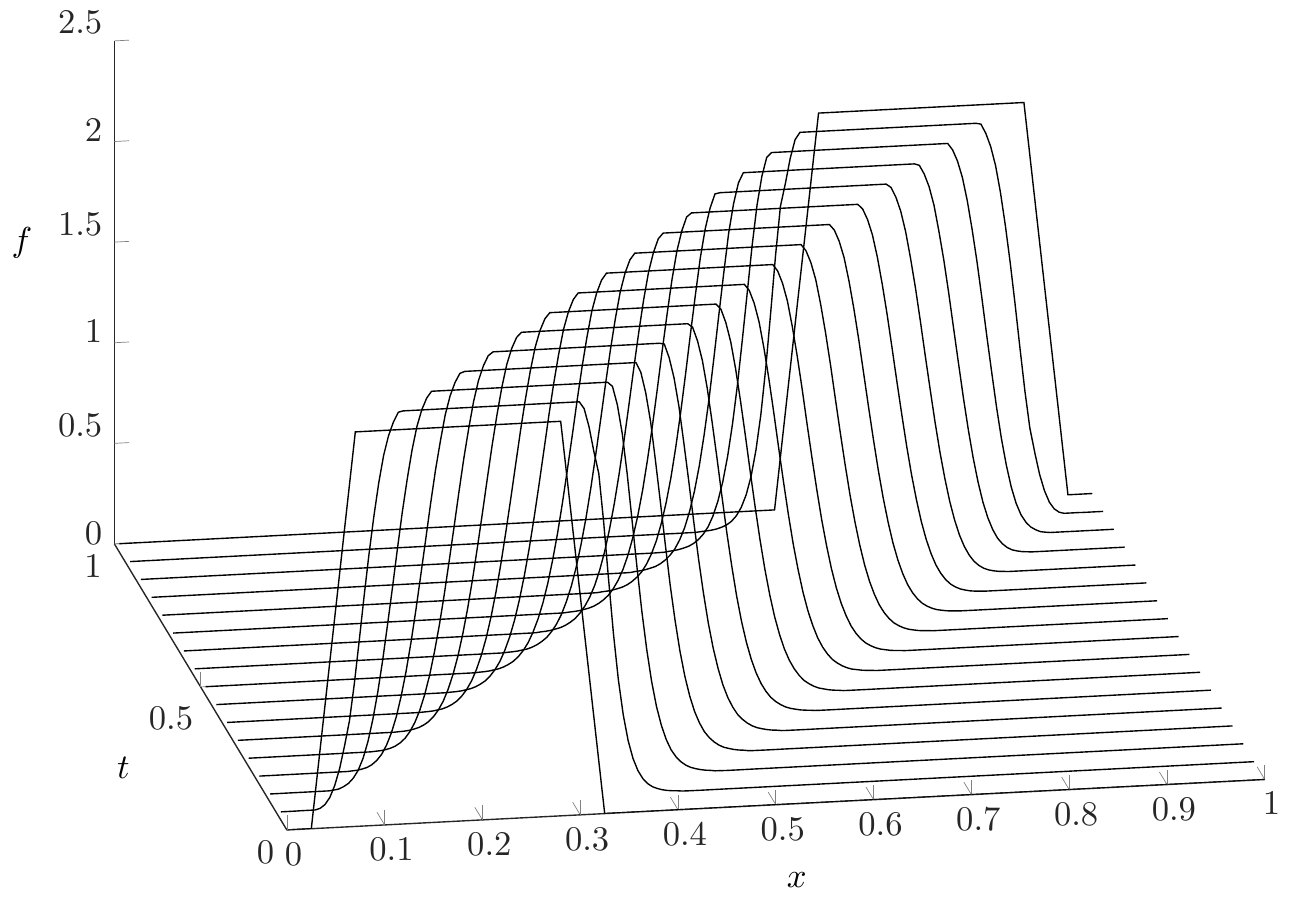}
\includegraphics[width=.45\textwidth]{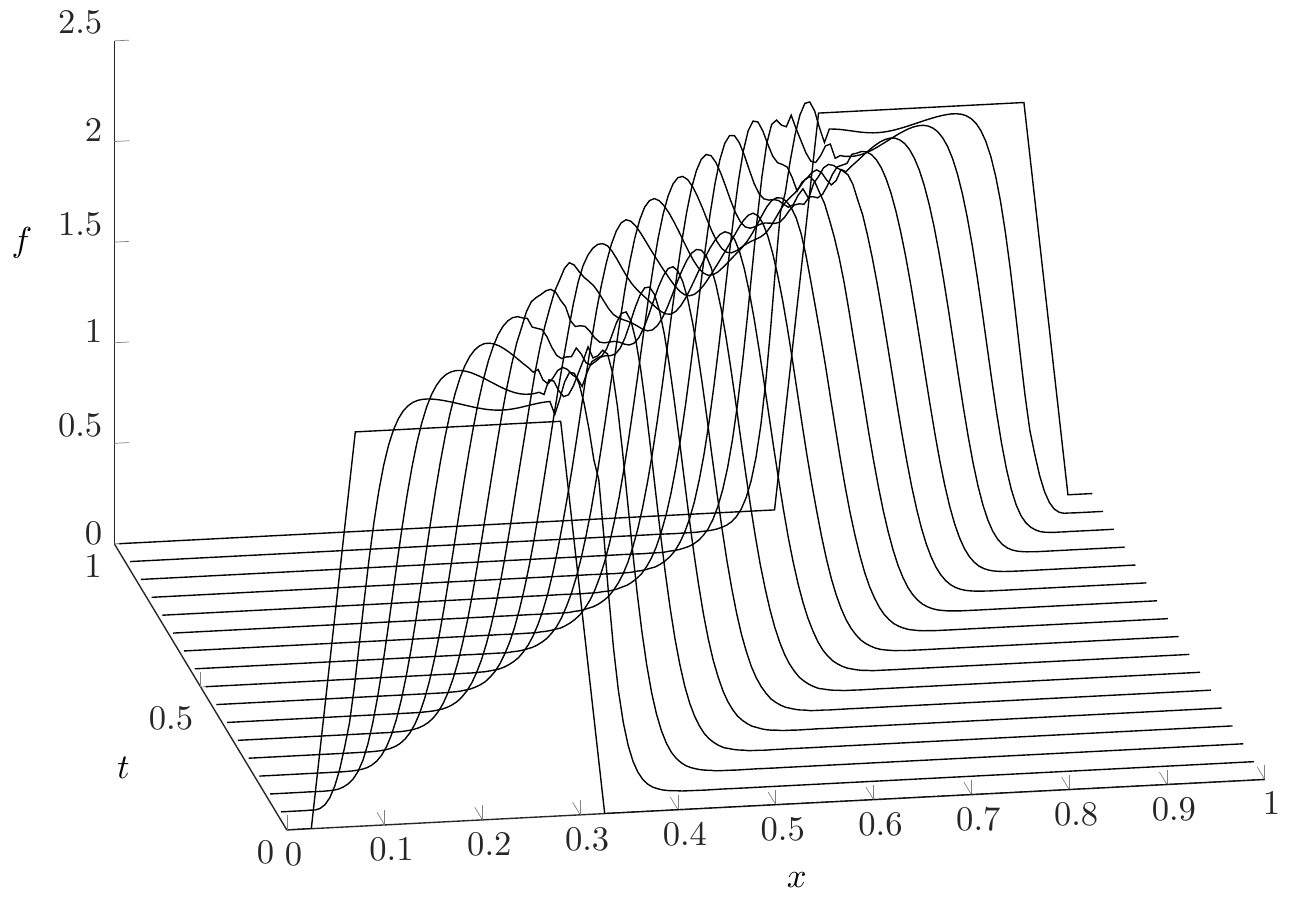}
\caption{\label{Fig:TVbump} Transport of a one-dimensional signal with sharp edges using a TV penalized functional \protect\eqref{TV_functional} (left), 
where $\gamma = 0.03$ and the result without TV regularization (right).}
\label{fig:TVbump}
\end{figure}
For the one-dimensional example this works well.
However, in higher dimensions one has to be more careful.
In Figure~\ref{Fig:TVbump_2D} a comparison of an isotropic and an an\-iso\-tropic TV regularizer is shown.
The isotropic one leads, as could be expected, to a rounding of the corners.
The anisotropic regularizer prefers horizontal and vertical edges. In this way, the shape of the object is preserved during the transport. Note that due to the smeared boundary the square appears to be smaller. 
To preserve the shape of  arbitrary transported objects, one would have to adjust the regularizer according to the direction of the edges.\\
The idea of penalizing TV terms for the transport can be found 
for gray-value images, e.g.\ in~\cite{Bru10,MRSS14}.
For image denoising a Wasserstein-TV model was successfully applied in \cite{BFS12,SS13x,BCDS13}.
\begin{figure*}
\centering
\includegraphics[width=.18\textwidth]{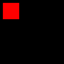}
\includegraphics[width=.18\textwidth]{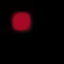}
\includegraphics[width=.18\textwidth]{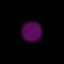}
\includegraphics[width=.18\textwidth]{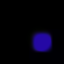}
\includegraphics[width=.18\textwidth]{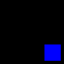}\\[1.5ex]
\includegraphics[width=.18\textwidth]{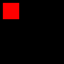}
\includegraphics[width=.18\textwidth]{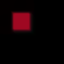}
\includegraphics[width=.18\textwidth]{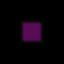}
\includegraphics[width=.18\textwidth]{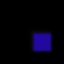}
\includegraphics[width=.18\textwidth]{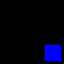}\\[1.5ex]
\includegraphics[width=.18\textwidth]{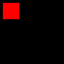}
\includegraphics[width=.18\textwidth]{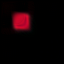}
\includegraphics[width=.18\textwidth]{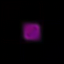}
\includegraphics[width=.18\textwidth]{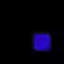}
\includegraphics[width=.18\textwidth]{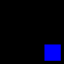}
\caption{\label{Fig:TVbump_2D} Transport of a two-dimensional square with sharp edges using a TV penalized functional, where $\gamma = 0.05$. First row: isotropic TV, second row: anisotropic TV, third row: no TV regularization.}
\label{fig:TVbump_2D}
\end{figure*}

\item[2)]
Using a barycentric approach the interpolation of microtextures in \cite{RPDB12} works also between more than two images.
So far this task can not be handled via the dynamic optimal transport approach.
One idea might it be to formulate a dynamic barycenter optimal transport problem or to use a multimarginal model, see e.g.\ Section 1.4.7. in~\cite{San15} and the references therein.
\end{enumerate}

	{\bf Acknowledgement:}
	Funding by the DFG within the Research Training Group 1932 is gratefully 
	acknow\-ledged.

\appendix
\section{Diagonalization of Structured Matrices} \label{sec:appA}
In the following we collect known facts on the eigenvalue decomposition 
of various difference matrices. For further information we refer, e.g., to  \cite{PS98,SM14}.
The following matrices $F_n$, $C_n$ and $S_n$ are unitary, resp., orthogonal matrices.
The Fourier matrix 
$$
F_n \vcentcolon= \sqrt{\tfrac{1}{n}} \left( \mathrm{e}^{\frac{-2\pi \imag jk}{n}} \right)_{j,k=0}^n
$$
diagonalizes circulant matrices, i.e., 
for $a \vcentcolon= (a_j)_{j=0}^{n-1} \in \mathbb R^n$ we have
\begin{align}\label{circ}
{\small
\left(
\begin{array}{llll}
a_0&a_{n-1}& \ldots& a_1\\
a_1&a_0&\ldots&a_2\\
\vdots & & \ddots& \vdots\\
a_{n-1} &a_1&\ldots&a_0
\end{array}
\right) 
}
&= \bar F_n \diag (\sqrt{n} F_n a) F_n\\& = F_n \diag (\sqrt{n} \, \bar F_n a) \bar F_n.
\end{align}
In particular it holds
\begin{align} \label{per}
\Delta_n^{\per} &\vcentcolon=  \frac{1}{n^2} (D_n^{\per})^\tT D_n^{\per} = \frac{1}{n^2} D_n^{\per} (D_n^{\per})^\tT\\
&=
{\scriptsize
\left(
\begin{array}{rrrrrrrr}
2 &-1&  &      &  &  &  &-1            \\
-1& 2&-1&                  \\
  &  &  &\ddots&  &        \\
  &  &  &      &  &-1&2 &-1 \\
-1&  &  &      &  &  &-1& 2
\end{array} 
\right)
}
= 
\bar F_n \diag (\dd^{\per}_n ) F_n 
\end{align}
with $\dd^{\per}_n \vcentcolon= \left( 4 \sin^2 \frac{k\pi}{n}\right)_{k=0}^{n-1}$.
The operator $\Delta_n^{\per}$ typically appears when solving  the one-dimensional Poisson equation with
periodic boundary conditions by finite difference methods.\\
The DST-I matrix
$$
S_{n-1} \vcentcolon= \sqrt{\tfrac{2}{n}} \left( \sin \frac{jk\pi}{n} \right)_{j,k=1}^{n-1}, 
$$
and the DCT-II matrix 
$$
C_n \vcentcolon= \sqrt{\tfrac{2}{n}} \left( \epsilon_j \cos\frac{j(2k+1)\pi}{2n} \right)_{j,k=0}^{n-1}
$$
with
$\epsilon_0 \vcentcolon= 1/\sqrt{2}$ and $\epsilon_j \vcentcolon= 1$, $j=1,\ldots,n-1$ 
are related by
\begin{align} \label{svd_diff}
D_n = S_{n-1} \left( 0 \, | \, \diag(\dd^{\zero}_{n-1})^\frac12 \right) C_n, 
\end{align}
where
$
\dd^{\zero}_{n-1} \vcentcolon= \left( 4 \sin^2 \frac{k\pi}{2n}\right)_{k=1}^{n-1}
$.
Further they diagonalize sums of certain symmetric Toeplitz and persymmetric Hankel matrices.
In particular it holds
\begin{align} \label{zero}
\Delta_{n-1}^{\zero} &\vcentcolon= \frac{1}{n^2} D_{n} D_{n} ^\tT \\
&=
{\scriptsize
\left(
\begin{array}{rrrrrrrr}
2 &-1&  &      &  &  &  &0 \\
-1& 2&-1&                  \\
  &  &  &\ddots&  &        \\
  &  &  &      &  &-1&2 &-1 \\
0&  &  &      &  &  &-1& 2
\end{array} 
\right) 
}
= 
S_{n-1} \diag(\dd^{\zero}_{n-1}) S_{n-1} 
\end{align}
and
\begin{align} \label{mirr}
\Delta_n^{\mirr} &\vcentcolon= \frac{1}{n^2} D_{n}^\tT D_{n} \\
&=
{\scriptsize
\left(
\begin{array}{rrrrrrrr}
1 &-1&  &      &  &  &  &0            \\
-1& 2&-1&                  \\
  &  &  &\ddots&  &        \\
  &  &  &      &  &-1&2 &-1 \\
0&  &  &      &  &  &-1& 1
\end{array} 
\right) 
}
= 
C_n^\tT \diag(\dd^{\mirr}_n) C_n 
\end{align}
with
$
\dd^{\mirr}_n \vcentcolon=  \begin{pmatrix} 0\\ \dd_{n-1}^{\zero} \end{pmatrix} = \left( 4\sin^2 \frac{j \pi}{2n} \right)_{j=0}^{n-1}
$.
The operators $\Delta_{n-1}^{\zero}$ and $\Delta_{n}^{\mirr}$ are related to the Poisson equation with zero boundary conditions
and mirror boundary conditions, respectively.
%
\section{Computation with Tensor Products} \label{sec:appB}
The tensor product (Kronecker product) of matrices
\begin{align*}
A
&=
\begin{pmatrix}
a_{1,1} & \cdots & a_{1,n}\\
\vdots & \cdots & \vdots \\
a_{m,1} & \cdots & a_{m,n}
\end{pmatrix}
\in \C^{m, n} \quad
\mathrm{and} \quad
B
=
\begin{pmatrix}
b_{1,1} & \cdots & b_{1,t}\\
\vdots & \cdots & \vdots \\
b_{s,1} & \cdots & b_{s,t}
\end{pmatrix}
\in \C^{s, t}
\end{align*}
is defined by 
\begin{equation*}
A \otimes  B
\vcentcolon=
\begin{pmatrix}
a_{1,1}  B & \cdots & a_{1,n} B\\
\vdots & \ddots & \vdots \\
a_{m,1} B & \cdots & a_{m,n} B
\end{pmatrix}
\in \mathbb C^{ms, nt}. 
\end{equation*}
The tensor product is associative and distributive with respect to the addition of matrices.
\begin{lemma}[Properties of Tensor Products]
${ }$
	\begin{enumerate}
		\item[{\rm i)}] $( A \otimes  B)^\tT =  A^\tT \otimes  B^\tT$ for $ A \in \C ^{m, n}$,
	$ B \in \C ^{s , t}$.
	
	Let $A,  C \in \C ^{m, m}$ and $ B,  D \in \C ^{n , n}$.
	Then the following holds:
	\item[{\rm ii)}] $( A \otimes  B)( C \otimes D) =  A  C \otimes  B  D$ for $A,  C \in \C ^{m, m}$ and $ B,  D \in \C ^{n , n}$.
	\item[{\rm iii)}] If $ A$ and $B$ are invertible, then  $ A \otimes  B$
		is also invertible and
		\begin{equation*}
		( A \otimes  B)^{-1} =  A^{-1} \otimes  B^{-1} \, .
		\end{equation*}
	\end{enumerate}
\end{lemma}
The tensor product is needed to establish the connection between images and their vectorized versions, 
i.e., we consider images $F\in\R^{n_1\times n_2}$ columnwise reshaped as
\begin{equation}
f
: =
\text{vec}(F) \in \R^{n_1 n_2}.
\end{equation}
Then the following relation holds true:
\begin{equation}
\text{vec}(A F B^\tT) = (B \otimes A) f.\label{vec_mat}
\end{equation}

\section{Proofs and Generalization of the Tensor Product Approach to 3D}
\label{sec:appC}

%
\begin{proof}[Proof of Proposition \ref{moore-penrose-constr}.]
By definition of $A$ and using \eqref{mirr}, \eqref{per},  we obtain for periodic boundary conditions
\begin{align}
A A^\tT 
&= I_P \otimes (D_N^{\per})^\tT D_N^{\per} + D_P^\tT D_P \otimes I_N\\
&= I_P \otimes  N^2 \Delta_N^{\per} + P^2 \Delta_P^{\mirr} \otimes I_N\\
&= (C_P^\tT \otimes \bar F_N) \; \diag( I_P \otimes  N^2 \dd_N^{\per}
 +  P^2 \dd_P^{\mirr} \otimes I_N) \; (C_P \otimes F_N).
\end{align}
Similarly we get with \eqref{mirr} for mirror boundary conditions
\begin{align}
A A^\tT  
&= I_P \otimes D_{N-1}^\tT D_{N-1} + D_{P}^\tT D_{P} \otimes I_{N-1} \\
&= I_P \otimes  N^2 \Delta_{N-1}^{\mirr} + P^2 \Delta_P^{\mirr} \otimes 
I_{N-1}\\
&= (C_P^\tT \otimes C_{N-1}^\tT) \; {\rm diag}( I_P \otimes  N^2 
\dd_{N-1}^{\mirr}
+  P^2 \dd_{N-1}^{\mirr} \otimes I_{N-1}) \; (C_P \otimes C_{N-1})
\end{align}
which finishes the proof. 
\end{proof}
%
\begin{proof}[Proof of Proposition \ref{lem:schur}.]
By definition of $A$ we obtain
\begin{align*}
\lambda A^\tT A + \frac{1}{\tau} I &= 
\left(
\begin{array}{cc}
\lambda D_{\rm m}^\tT D_{\rm m} + \frac{1}{\tau} I&\lambda D_{\rm m}^\tT D_{\rm f}\\
\lambda D_{\rm f}^\tT D_{\rm m}&\lambda D_{\rm f}^\tT D_{\rm f}  + \frac{1}{\tau} I
\end{array}
\right)
=\vcentcolon
\left(
\begin{array}{cc}
X     & Y\\
Y^\tT & Z
\end{array} 
\right)
\end{align*}
so that the inverse can be written by the help of the Schur complement
$$
S \vcentcolon= Z - Y^\tT X^{-1} Y 
$$
as
$$
\left(
\begin{array}{cc}
X     & Y\\
Y^\tT& Z
\end{array}
\right)^{-1}
= 
\left(
\begin{array}{ll}
I    & - X^{-1}Y\\
0 & I
\end{array} 
\right)
\left(
\begin{array}{cc}
X^{-1}   & 0\\
0 & S^{-1}
\end{array} 
\right)
\left(
\begin{array}{cc}
I    & 0\\
-Y^\tT X^{-1}& I
\end{array} 
\right).
$$
By \eqref{per} and \eqref{zero} we have with $D \in \{D_N^{\per},D_N\}$ that
\begin{align}
X^{-1} &= \left(\lambda D_{\rm m}^\tT D_{\rm m} + \tfrac{1}{\tau} I\right)^{-1} 
= \left(I_P \otimes \lambda N^2 D D^\tT +\tfrac{1}{\tau} I\right)^{-1}  \\
&= I_P \otimes (\lambda N^2 D D^\tT +\tfrac{1}{\tau} I)^{-1} \\
&= 
\begin{cases}
I_P \otimes S_{N-1} \diag(\lambda N^2 \dd_{N-1}^{\zero} + \tfrac{1}{\tau}) ^{-1} S_{N-1} & \text{mirror boundary},\vspace{0.25cm}\\
I_P \otimes F_N \diag(\lambda N^2 \dd_N^{\per} + \tfrac{1}{\tau}) ^{-1} \bar F_N
& \text{periodic boundary}.
\end{cases}
\end{align}
The Schur complement reads as
\begin{align}
S 
&= (\lambda D_{\rm f}^\tT D_{\rm f}  + \tfrac{1}{\tau} I) - \lambda^2 D_{\rm f}^\tT D_{\rm m} X^{-1} D_{\rm m}^\tT D_{\rm f}\\
&= (\lambda D_P D_P^\tT \otimes I_N + \tfrac{1}{\tau} I)  - 
\lambda^2 (D_P \otimes D^\tT) \big(  I_P \otimes (\lambda N^2 D D^\tT 
+\tfrac{1}{\tau} I)^{-1} \big) (D_P^\tT \otimes D)\\ 
&= (\lambda D_P D_P^\tT \otimes I_N + \tfrac{1}{\tau} I)  - 
\lambda^2 \big( D_P D_P^\tT \otimes D^\tT (\lambda D D^\tT + \tfrac{1}{\tau} I_N)^{-1} D \big)\\
&= \lambda D_P D_P^\tT \otimes \left(I_N - \lambda  D^\tT (\lambda D D^\tT + \tfrac{1}{\tau} I_N)^{-1} D \right) + \tfrac{1}{\tau} I.
\end{align}
By \eqref{circ} we have
$$(D_N^{\per})^\tT = N F_N \diag(-1 + \e^{+2\pi \imag k/N})_k \bar F_N$$
and 
$$D_N ^{\per}= N  F_N \diag(-1 + \e^{-2\pi \imag k/N})_k \bar F_N$$
so that we obtain for periodic boundary boundary conditions
\begin{align}
I_N - \lambda  (D_N^{\per})^\tT (\lambda D_N^{\per} (D_N^{\per})^\tT + \tfrac{1}{\tau} I_N)^{-1} D_N^{\per}
= F_N \diag \left(1+ \tau \tfrac{1}{\lambda N^2} \dd^{\per}_N \right)^{-1} \bar F_N.
\end{align}
Therewith it follows with \eqref{zero}
\begin{align}
S 
&=  
S_{P-1} \diag (\lambda P^2 \dd^{\zero}_{P-1}) S_{P-1}  \otimes F_N 
\diag \left(1+ \tau \tfrac{1}{\lambda N^2} \dd^{\per}_N \right) ^{-1} 
\bar F_N  + \tfrac{1}{\tau} I\\
&= 
(S_{P-1} \otimes F_N) 
\diag \big(\lambda P^2 \dd^{\zero}_{P-1} \otimes(1+ \tau \tfrac{1}{\lambda N^2} \dd^{\per}_N )^{-1} + \tfrac{1}{\tau} \big) 
(S_{P-1} \otimes \bar F_N)
\end{align}
which yields the assertion for $S^{-1}$ in the periodic case.\\
For mirror boundary conditions we compute using \eqref{svd_diff}
\begin{align}
S &=
(S_{P-1} \otimes C_N^\tT) 
\diag \big(\lambda P^2 \dd^{\zero}_{P-1} \otimes(1+ \tau \tfrac{1}{\lambda N^2} \dd^{\mirr}_N )^{-1} + \tfrac{1}{\tau}\big) 
(S_{P-1} \otimes C_N)
\end{align}
and inverting this matrix finishes the proof.
\end{proof}
\vspace{0.2cm}

\noindent
{\bf Discretization for three spatial dimensions + time}:
  For RGB images of size $N_1 \times N_2 \times N_3$, where $N_3 = 3$, we have to work in three  spatial dimensions.
Setting $N \coloneqq (N_1,N_2,N_3)$, $j \coloneqq (j_1,j_2,j_3)$ and defining 
the quotient $\tfrac{j}{N}$ componentwise we obtain
\begin{itemize}
\item $f_i = \left(f_i( \tfrac{j-1/2}{N}) \right)_{j=(1,1,1)}^{N} \in \mathbb R^{N_1,N_2,N_3}$, $i= 0,1$,
\item $f = \left( f( \tfrac{j-1/2}{N}, \tfrac{k}{P} ) \right)_{j={(1,1,1)},k=1}^{N,P-1} \in \mathbb R^{N_1,N_2,3,P-1}$,
\item $m = (m_1, m_2, m_3)$, with
\begin{align*}
	&\left(m_1(\tfrac{j_1}{N_1},\tfrac{j_2-1/2}{N_2},\tfrac{j_3-1/2}{3},\tfrac{k-1/2}{P})\right)_{j_1=1,j_2=1,j_3=1,k=1}^{N_1-1,N_2,3,P}\in \mathbb R^{N_1-1,N_2,3,P},\\[1.5ex]
	&\left(m_2(\tfrac{j_1-1/2}{N_1},\tfrac{j_2}{N_2},\tfrac{j_3-1/2}{3},\tfrac{k-1/2}{P})\right)_{j_1=1,j_2=1,j_3=1,k=1}^{N_1,N_2-1,3,P}\in \mathbb R^{N_1,N_2-1,3,P},\\[1.5ex]
	&\left(m_3(\tfrac{j_1-1/2}{N_1},\tfrac{j_2-1/2}{N_2},\tfrac{j_3}{3},\tfrac{k-1/2}{P})\right)_{j_1=1,j_2=1,j_3=0,k=1}^{N_1,N_2,2,P}\in \mathbb R^{N_1,N_2,3,P}.
\end{align*}

\end{itemize}
In the definition of $m$ we take the periodic boundary  for the third spatial 
direction into account. Analogously as in the one-dimensional case, when 
reshaping $m$ and $f$ into long vectors, the interpolation and differentiation 
operators can be written using tensor products. For the interpolation operator 
we have
\begin{align*}
S_{\text{m}}m &=	\begin{pmatrix}
	(I_P \otimes I_3\otimes I_{N_2} \otimes S_{N_1}^\tT) m_1\\
	(I_P \otimes I_3\otimes S_{N_2}^\tT \otimes I_{N_1}) m_2\\
	(I_P \otimes S_3^\tT\otimes I_{N_2} \otimes I_{N_1}) m_3\end{pmatrix}
\end{align*}
and
\begin{align*}
S_{\text{f}}f =
	(S_P^\tT \otimes I_3\otimes I_{N_2} \otimes I_{N_1}) f,
\end{align*}
which means, that $S_{\text{m}}m$ computes the average of $m_i$ with respect to the $i$-th coordinate, $i=1,2,3$, and $S_{\text{f}}f$ computes the average of $f$ with respect to the time variable. 
Similarly we generalize the difference operator.
Then, reordering $f$ and $m$ into large vectors, the matrix form of the operator $A$ is
\begin{multline}
	A = 
	\big( 
	I_P\otimes I_3\otimes I_{N_2}\otimes D_{N_1}^\tT \, | \, I_P\otimes I_3\otimes D_{N_2}^\tT\otimes I_{N_1}\, | \,\\
	I_P\otimes (D_3^{\per})^\tT\otimes I_{N_2}\otimes I_{N_1} \, | \, D_P\otimes I_3\otimes I_{N_2}\otimes I_{N_1} 
	\big)\label{diff_operator}
\end{multline}
so that $AA^\tT$ reads as 
\begin{align*}
A A^\tT	 
&= I_P\otimes I_3\otimes I_{N_2}\otimes D_{N_1}^\tT D_{N_1}+I_P\otimes I_3\otimes D_{N_2}^\tT D_{N_2}\otimes I_{N_1}\\
& \ \ \ \ +I_P\otimes (D_3^{\per})^\tT(D_3^{\per})\otimes I_{N_2}\otimes I_{N_1} +D_P^\tT D_P\otimes I_3\otimes I_{N_2}\otimes I_{N_1}\\
&= \left(C^\tT_P \otimes  \bar{F_3} \otimes C_{N_2-1}^\tT\otimes C_{N_1-1}^\tT\right)\diag(\dd)\left(C_P \otimes  F_3 \otimes C_{N_2-1}\otimes C_{N_1-1}\right),
\end{align*}
where 
\begin{align*}
	\dd &\coloneqq I_{3 P (N_2-1)} \otimes N_1^2 
	\diag(\dd_{N_1-1}^{\text{mirr}}) + I_{3P}\otimes {N_2}^2 \diag(\dd_{{N_2}-1}^{\text{mirr}})\otimes 
	 I_{N_1-1}\\ 
	 & \ \ \ \ + I_P \otimes 3^2 \diag(\dd_{3}^{\text{per}}) \otimes I_{(N_2-1)(N_1-1)}+ P^2 \diag(\dd_{P}^{\text{mirr}})\otimes I_{3P(N_2-1)(N_1-1)}.
\end{align*}
For the three-dimensional spatial setting we have to solve a four-dimensional Poisson equation, which can be handled separately in each dimension.
For the constrained problem, this can be computed using fast cosine and Fourier transforms with a complexity of $\mathcal{O}(N_1 N_2 P\log(N_1 N_2 P))$.
\bibliography{OT}
\bibliographystyle{abbrv}

\end{document}